\documentclass[12pt,a4paper]{article}
\usepackage[T1]{fontenc}
\usepackage{lmodern}
\usepackage[utf8]{inputenc}
\usepackage{amsthm}
\usepackage{amsmath,amssymb}
\usepackage{graphics}
\usepackage{lscape,graphicx}
\usepackage{enumitem}
\usepackage{amsfonts}
\usepackage[T1]{fontenc}
\usepackage{longtable}
\usepackage{authblk}
\usepackage{lipsum}
\usepackage{epsfig}
\usepackage{todonotes}
\usepackage{float}
\usepackage{hyperref}
\usepackage{scalerel}
\usepackage{comment}
\usepackage[normalem]{ulem}
\usepackage{multicol}
\usepackage{fancyhdr}
\usepackage[english]{babel}
\usepackage{mathrsfs}
\usepackage{tocloft}
\usepackage{xcolor}
\usepackage{titletoc}
\usepackage{mathtools}
\usepackage[toc,page]{appendix}

\addto{\captionsenglish}{}
\makeatletter
\def\thm@space@setup{%
  \thm@preskip=1cm plus .5cm minus .5cm
  \thm@postskip=.5cm plus .6cm minus .5cm 
}

\makeatother

\newtheorem{thm}{Theorem}
\newtheorem{fact}{Fact}
\newtheorem{lma}{Lemma}

\newtheorem{cor}{Corollary}

\newtheorem{rmk}{Remark}
\newtheorem{ex}{Example}
\numberwithin{thm}{section}
\numberwithin{lma}{section}
\numberwithin{dfn}{section}
\numberwithin{cor}{section}
\numberwithin{rmk}{section}
\numberwithin{prop}{section}

\def\mfd{{\mathfrak d}}
\def\mfm{{\mathfrak m}}
\def\mfc{{\mathfrak C}}
\def\mfa{{\mathfrak A}}
\def\mcm{{\mathcal M}}
\def\mcp{{\mathcal P}}
\def\mcs{{\mathcal S}}
\def\mcA{{\mathcal A}}
\def\mfp{{\mathfrak p}}

\newcommand*{\thmref}[1]{Theorem~\ref{#1}}

\newcommand*{\lmaref}[1]{Lemma~\ref{#1}}
\newcommand*{\rmkref}[1]{Remark~\ref{#1}}

\title{Generalizations of Erd\H{o}s-Kac theorem with applications}

\date{}

\begin{document}





  



\author[1]{Sourabhashis Das}
\affil[1]{Corresponding author, Department of Pure Mathematics, University of Waterloo, 200 University Avenue West, Waterloo, Ontario, Canada, N2L 3G1, \textit{s57das@uwaterloo.ca}.}
\author[2]{Wentang Kuo}
\affil[2]{Department of Pure Mathematics, University of Waterloo, 200 University Avenue West, Waterloo, Ontario, Canada, N2L 3G1, \textit{wtkuo@uwaterloo.ca}.}
\author[3]{Yu-Ru Liu}
\affil[3]{Department of Pure Mathematics, University of Waterloo, 200 University Avenue West, Waterloo, Ontario, Canada, N2L 3G1, \textit{yrliu@uwaterloo.ca}.}

%
%
%

\maketitle 

\begin{abstract}
\footnotetext{The research of W. Kuo and Y.-R. Liu is supported by the Natural Sciences and Engineering Research Council of Canada (NSERC) grants.}Let $\omega(n)$ denote the number of distinct prime factors of a natural number $n$. In 1940, Erd\H{o}s and Kac established that $\omega(n)$ obeys the Gaussian distribution over natural numbers, and in 2004, the third author generalized their theorem to all abelian monoids. In this paper, we extend her theorem to any subsets of an abelian monoid satisfying some additional conditions, and apply this result to the subsets of $h$-free and $h$-full elements. We study generalizations of several arithmetic functions, such as the prime counting omega functions and the divisor function in a unified framework. Finally, we apply our results to number fields, global function fields, and geometrically irreducible projective varieties, demonstrating the broad relevance of our approach.
\end{abstract}

\section{Introduction}\label{intro}

Let $n$ be a natural number and let $\omega(n)$ count its number of distinct prime divisors. In \cite{ErdosKac}, Erd\H{o}s and Kac used probabilistic means to establish that $\omega(n)$ obeys the Gaussian distribution with mean $\log \log n$ and variance $\log \log n$ over natural numbers. In particular, for any $\gamma \in \mathbb{R}$, they proved
$$\lim_{x \rightarrow \infty} \frac{1}{x}  \left| \left\{ n \leq x \ : \ n \geq 3, \ \frac{\omega(n) - \log \log n}{\sqrt{\log \log n}} \leq \gamma \right\} \right| = \Phi(\gamma),$$
where 
\begin{equation}\label{phi(a)}
    \Phi(\gamma) = \frac{1}{\sqrt{2 \pi}} \int_{-\infty}^\gamma e^{-u^2/2} \ du,
\end{equation}
and where $|\mcs|$ denotes the cardinality of the set $\mcs$. Following their work, various approaches to the Erd\H{o}s-Kac theorem have been pursued. For example, in \cite{mmp}, Murty, Murty, and Pujahari proved an all-purpose Erd\H{o}s-Kac theorem which applies to diverse settings. In \cite{dkl3}, the authors provided a generalization of the Erd\H{o}s-Kac theorem in number fields. Moreover, the third author, in \cite{liu}, provided a generalization of the Erd\H{o}s-Kac theorem over any countably generated abelian monoid. We extend this work to provide another generalization of the Erd\H{o}s-Kac theorem over any subset of such abelian monoids satisfying some additional conditions. 

Let $\mcp$ be a countable set of elements with a map
$$N: \mcp \rightarrow \mathbb{Z}_{>1}, \quad \mfp \mapsto N(\mfp).$$
We call this map the Norm map. Let $\mcm$ be a free abelian monoid generated by elements of $\mcp$. For each $\mfm \in \mcm$, we write
$$\mfm = \sum_{\mfp \in  \mcp} n_\mfp(\mfm) \mfp,$$
with $n_\mfp(\mfm) \in \mathbb{Z}_{>0} \cup \{ 0 \}$ and $n_\mfp(\mfm) =0$ for all but finitely many $\mfp$. We extend the map $N$ to $\mcm$ as the following:
\begin{align*}
    N : \mcm & \rightarrow \mathbb{Z}_{>0} \\
    \mfm = \sum_{\mfp \in \mcp} n_\mfp(\mfm) \mfp & \longmapsto N(\mfm) := \prod_{\mfp \in \mcp} N(\mfp)^{n_\mfp(\mfm)}. 
\end{align*}
Thus, $N$ can be extended to a monoid homomorphism from $(\mcm,+)$ to $(\mathbb{Z}_{>0},\cdot)$. Let $X$ be a countable subset of $\mathbb{Q}$ that contains the image $\text{Im}(N(\mcm))$ with an extra condition: if $x_1,x_2 \in X$, the fraction $x_1/x_2$ belongs to $X$, too. Without loss of generality, we assume $X = \mathbb{Q}$ or $X = \{ q^z :  z \in \mathbb{Z} \}$ for some $q \in \mathbb{Z}_{>0}$. For interested readers, the details behind $X$ being limited to these two choices are presented in \cite[Theorem 2]{liu}.

Given $\mcp$, $\mcm$, $X$, and for sufficiently large $x \in X$, we assume that the following condition hold: 
\begin{equation}\label{star}
    \mcm(x) := \sum\limits_{\substack{\mfm \in \mcm \\ N(\mfm) \leq x}} 1 = \kappa x + O(x^\theta), \quad \text{ for some } \kappa > 0 \text{ and } 0 \leq \theta < 1. \tag{$\star$}
\end{equation}
Let $\mathcal{\mcs}$ be a subset of infinitely many elements in $\mcm$. For $x \in X$, $x >1$, we define
$$\mcs(x) = \{ \mfm \in \mcs \ : \ N(\mfm) \leq x \}.$$
We assume that $\mcs$ satisfies the following condition:
\begin{equation}\label{Scondition}
    |\mcs(x^{1/2})| = o(|\mcs(x)|) \quad \text{as } x \rightarrow \infty.
\end{equation}
Let $\kappa' > 0$, $0 < \alpha < 1$, and $\gamma \in \mathbb{R}$. Note that, $|S(x)| \ll x$, \eqref{Scondition} is satisfied when $\mcs(x) \sim \kappa' x^\alpha (\log x)^\gamma$, and \eqref{Scondition} fails if $\mcs(x) \sim \kappa' (\log x)^\gamma$. Thus, $\mcs(x) \sim \kappa' x^\alpha (\log x)^\gamma$ would be a good choice for the asymptotic size of $\mcs(x)$ so that \eqref{Scondition} holds.

Let $f$ be a map from $\mcs$ to $\mcm$. For each $\ell \in \mcp$, we write
$$\frac{1}{|\mcs(x)|} \left| \{ \mfm \in \mcs(x) \ : \ n_\ell(f(\mfm)) \geq 1 \} \right| = \lambda_\ell + e_\ell(x),$$
where $\lambda_\ell$ denotes the main term (and is independent of $x$) and $e_\ell = e_\ell(x)$ is the error term. For any sequence of distinct elements $\ell_1, \ell_2, \cdots, \ell_u \in \mcp$, we write  
$$\frac{1}{|\mcs(x)|} \left| \{ \mfm \in \mcs(x) \ : \ n_{\ell_i}(f(\mfm)) \geq 1 \text{ for all } i = 1, \cdots, u \} \right|= \lambda_{\ell_1} \cdots \lambda_{\ell_u} + e_{\ell_1 \cdots \ell_u}(x).$$
In this article, we will use $e_{\ell_1 \cdots \ell_u}$ to abbreviate $e_{\ell_1 \cdots \ell_u}(x)$.

Suppose there exists a $\beta$ with $0 < \beta \leq 1$ and $y = y(x) < x^\beta$ such that the following conditions are satisfied:
   \begin{enumerate}
       \item[(a)] $\left|\{ \ell \in \mcp \ : \ N(\ell) > x^\beta, n_\ell(f(\mfm)) \geq 1 \} \right| = O_\beta(1)$ for each $\mfm \in \mcs(x)$. Here, $O_{Y}$ denotes that the big-O constant depends on the variable set $Y$. 
       \item[(b)] $\sum_{y < N(\ell) \leq x^\beta} \lambda_\ell = o \left( (\log \log x)^{1/2} \right)$.
       \item[(c)] $\sum_{y < N(\ell) \leq x^\beta} |e_\ell| = o \left( (\log \log x)^{1/2} \right)$.
       \item[(d)] $\sum_{N(\ell) \leq y} \lambda_\ell = \log \log x + o \left( (\log \log x)^{1/2} \right)$.
       \item[(e)] $\sum_{N(\ell) \leq y} \lambda_\ell^2 = o \left( (\log \log x)^{1/2} \right)$.
       \item[(f)] For $r \in \mathbb{Z}_{>0}$, let $u$ be any integer picked from $\{ 1, 2, \cdots, r \}$. We have
       $$\sum{\vphantom{\sum}}'' |e_{\ell_1 \cdots \ell_u}| = o \left( (\log \log x)^{-r/2} \right), $$
       where $\sum{\vphantom{\sum}}''$ extends over all $u$-tuples $(\ell_1,\ell_2,\cdots,\ell_u)$ with $N(\ell_i) \leq y$ for all $i \in \{ 1, 2, \cdots, u\}$ and all $\ell_i$'s are distinct.
   \end{enumerate}
For each $\mfm \in \mcm$, we define
$$\omega(\mfm) = \sum_{\substack{\mfp \in \mcp \\ n_\mfp(\mfm) \geq 1}} 1,$$
the number of elements of $\mcp$ that generates $\mfm$, counted without multiplicity. Using this definition and the above conditions, we obtain the following subset generalization of the Erd\H{o}s-Kac theorem:
\begin{thm}\label{yurugen}
Let $\mcp$, $\mcm$, and $X$ satisfy \eqref{star}. Let $\mcs$ be a subset of $\mcm$. For any $x \in X$, let $\mcs(x)$ be the set of elements in $\mcs$ with norm less than or equal to x. Assume that $\mcs$ satisfies condition \eqref{Scondition}. Let $f : \mcs \rightarrow \mcm$. Suppose there exists a $\beta$ with $0 < \beta \leq 1$ and $y = y(x) < x^\beta$ such that the conditions (a) to (f) above hold. 
Then for $\gamma \in \mathbb{R}$, we have
$$\lim_{x \rightarrow \infty} \frac{1}{|\mcs(x)|} \left| \left\{ \mfm \in \mcs(x) \ : \ N(\mfm) \geq 3, \ \frac{\omega(f(\mfm)) - \log \log N(\mfm)}{\sqrt{\log \log N(\mfm)}} \leq \gamma\right\} \right| = \Phi(\gamma),$$
where $\Phi(\gamma)$ is defined in \eqref{phi(a)}.
\end{thm}
We list some well-studied applications of this general setting.
\begin{ex}
    Let $\mcp$, $\mcm$, and $X$ satisfy \eqref{star}. Let $\mcs = \mcm$ and $f$ be the identity map. Then, by \thmref{yurugen}, we recover the Erd\H{o}s-Kac theorem over abelian monoids as studied by the third author in \cite[Theorem 1]{liu}.
\end{ex}
\begin{ex}
Let $K/\mathbb{Q}$ be a number field. Let $\mcp$ be the set of prime ideals of $K$. For an ideal $I$ of $K$, let the map $N(I)$ be the absolute norm of $I$. Let $\mcs=\mcm$, $X= \mathbb{Q}$, and $f$ be the identity map. If $K = \mathbb{Q}$, by \thmref{yurugen}, we recover the classical Erd\H{o}s-Kac theorem. Moreover, in \cite[Theorems 1.3 \& 1.4]{dkl3}, the authors apply \thmref{yurugen} to establish the Erd\H{o}s-Kac theorems over the subsets of $h$-free and $h$-full ideals in any number field $K$.
\end{ex}

In the following part, we discuss several other instances where \thmref{yurugen} can be applied. Note that, moving forward, we shall always assume that $\mathcal{P}, \ \mcm$, and $X$ satisfy the condition \eqref{star}.

For a non-zero element $\mathfrak{m} \in \mathcal{M}$, let the prime element factorization of $\mathfrak{m}$ be given as
\begin{equation*}
\mathfrak{m} = s_1 \mfp_1 + \cdots + s_r \mfp_r,
\end{equation*}
where $\mfp_i'$s are its distinct prime elements and $s_i'$s are their respective non-zero multiplicities. Here, 
$$N(\mfm) = N(\mfp_1)^{s_1} \cdots N(\mfp_r)^{s_r}.$$ 
Let $h \geq 2$ be an integer. We say $\mathfrak{m}$ is an $h$\textit{-free} element if $s_i \leq h-1$ for all $i \in \{1, \cdots, r\}$, and we say $\mathfrak{m}$ is an $h$\textit{-full} element if $s_i \geq h$ for all $i \in \{1, \cdots, r\}$. Let $\mathcal{S}_h$ denote the set of $h$-free elements and $\mathcal{N}_h$ denote the set of $h$-full elements. The distributions of these elements are well-studied in the literature, and to demonstrate this we introduce some terminologies.

Let the generalized $\zeta$-function which is an analog of the classical Riemann $\zeta$-function be given as:
\[
\zeta_\mcm(s) := \sum_{\substack{\mathfrak{m}}} \frac1{(N(\mathfrak{m}) )^s} 
= \prod_{\mfp} \Big( 1- N (\mfp)^{-s}\Big)^{-1}
\ \text{for } \Re (s) >1,
\]
where $\mathfrak{m}$ and $\mfp$ respectively range through the non-zero elements in $\mcm$ and the prime elements in $\mcp$. The absolute convergence of the above series is explained in \cite{dkl5}.

Let $x \in X$ and let $\mathcal{S}_h(x)$ denote the set of $h$-free elements in $\mcm$ with norm less than or equal to $x$. Since condition \eqref{star} satisfies \cite[Chapter 4, Axiom A]{jk2}, thus, by \cite[Chapter 4, Proposition 5.5]{jk2}, we have:
\begin{equation}\label{hfreeidealcount}
    |\mathcal{S}_h(x)| = \frac{\kappa}{\zeta_\mcm(h)} x + O_h \big( R_{\mathcal{S}_h}(x) \big),
\end{equation}
where
   \begin{equation}\label{RSh(x)}
    R_{\mathcal{S}_h}(x) = \begin{cases}
    x^\theta  & \text{ if } \frac{1}{h} < \theta, \\
    x^\theta (\log x) & \text{ if } \frac{1}{h} = \theta, \\
    x^{\frac{1}{h}} & \text{ if } \frac{1}{h} > \theta.
\end{cases}
    \end{equation}

\begin{rmk}\label{remark1}
    In this paper, for convenience, we shall use $R_{\mathcal{S}_h}(x) \ll x^{\tau}$ for some $\tau < 1$, which is evident from the above result. 
\end{rmk}

Let $\gamma_h$ be a constant given by
\begin{equation}\label{gammahk}
    \gamma_{\scaleto{h}{4.5pt}} = \gamma_{\scaleto{h,\mcm}{5.5pt}} := \prod_\mfp \left( 1 + \frac{N(\mfp) - N(\mfp)^{1/h}}{N(\mfp)^2 \left( N(\mfp)^{1/h} - 1 \right)}\right).
\end{equation}

Let $\mathcal{N}_h(x)$ denote the set of $h$-full elements in $\mcm$ with norm less than or equal to $x$. For the distribution of $h$-full elements, we have (see \cite[Theorem 1.1]{dkl5})
\begin{equation}\label{hfullideals}
    |\mathcal{N}_h(x)| = \kappa \gamma_{\scaleto{h}{4.5pt}} x^{1/h} + O_h \big( R_{\mathcal{N}_h}(x) \big),
\end{equation}
where $\gamma_{\scaleto{h}{4.5pt}}$ is the constant defined in \eqref{gammahk},
and where
\begin{equation}\label{E2(x)}
    R_{\mathcal{N}_h}(x) = 
    \begin{cases}
        x^{\theta/h} & \text{ if } \frac{h}{h+1} < \theta, \\
        x^{\frac{1}{h+1}} (\log x) & \text{ if } \frac{h}{h+i} = \theta \text{ for some } i \in \{ 1, \cdots, h-1 \},  \\
        x^{\frac{1}{h+1}} & \text{ if } \frac{h}{h+1} > \theta \ \& \ \frac{h}{h+i} \neq \theta \text{ for any } i \in \{ 1, \cdots, h-1 \}.
    \end{cases}
\end{equation}

\begin{rmk}\label{remark2}
    In this paper, again for convenience, we shall use $R_{\mathcal{N}_h}(x) \ll x^{\nu/h}$ for some $\nu < 1$. 
\end{rmk}


Let $x \in X$. Let $\mathfrak{A}$ and $\mathfrak{B}$ be constants defined as
\begin{equation}\label{A}
    \mfa:= \lim_{x \rightarrow \infty} \left( \sum_{\substack{\mfp \in \mathcal{P} \\ N(\mfp) \leq x}} \frac{1}{N(\mfp)} - \log \log x \right).
\end{equation}
The existence of the constant $\mathfrak{A}$ is explained in \cite[Lemma 2]{liuturan}. We define the constants
\begin{equation}\label{C1}
    \mathfrak{C}_1 :=  \mfa - \sum_{\mfp} \frac{N(\mfp)-1}{N(\mfp)(N(\mfp)^h -1)},
\end{equation}
Let $\mathfrak{L}_h(r)$ be the convergent sum defined for $r > h$ as
\begin{equation}\label{lhr}
     \mathfrak{L}_h(r) := \sum_\mfp \frac{1}{N(\mfp)^{(r/h)-1} \left( N(\mfp) - N(\mfp)^{1-1/h} + 1 \right)},
\end{equation}
and let
\begin{equation}\label{D1}
    \mathfrak{D}_1 :=  \mathfrak{A} - \log h + \mathfrak{L}_h(h+1) - \mathfrak{L}_h(2h).
\end{equation}

For the distribution of $\omega(\mfm)$ over $h$-free and $h$-full elements, in \cite[Theorems 1.2 \& 1.3]{dkl5}, we proved
$$\sum_{\mathfrak{m} \in \mathcal{S}_h(x)} \omega(\mfm) = \frac{\kappa}{\zeta_\mcm(h)} x \log \log x + 
\frac{\kappa \mfc_1}{\zeta_\mcm(h)} x
+ O_h \left( \frac{x}{\log x}\right),$$
and
\begin{align*}
\sum_{\mfm \in \mathcal{N}_h(x)} \omega(\mfm) & = \kappa \gamma_{h} x^{1/h} \log \log x + \kappa \gamma_h \mathfrak{D}_1 x^{1/h} + O_h \left( \frac{x^{1/h}}{\log x} \right).
\end{align*}

This proves that $\omega(\mfm)$ has average order $\log \log N(\mfm)$ over $h$-free and over $h$-full elements. Using the study of moments, in \cite{dkl5}, we showed that $\omega(\mfm)$ has normal order $\log \log (\mfm)$ over $h$-free and over $h$-full elements. In this paper, we prove that $\omega(\mfm)$ obeys the Gaussian distribution over $h$-free and over $h$-full elements as well. We achieve this as applications of the following theorem which is a consequence of \thmref{yurugen}:

\begin{thm}\label{less-general}
Let $\mcp$, $\mcm$, and $X$ satisfy condition \eqref{star}. Let $\mcs$ be a subset of $\mcm$. For any $x \in X$, let $\mcs(x)$ be the set of elements in $\mcs$ with norm less than or equal to x. Let $f : \mcs \rightarrow \mcm$. Suppose $\mcs$ satisfy
$$|\mathcal{S}(x)| = C_\beta x^{\beta} + O_\beta \big( x^{\xi \beta} \big),$$
for some $0 < \beta \leq 1$, for some $0 \leq \xi < 1$, and for some fixed constant $C_\beta$. Additionally, for a fixed prime ideal $\mfp$, we assume that the set
$$\mathcal{S}_\mfp(x) := \left\{ \mathfrak{m} \in \mcs(x) \ : \ n_\mfp (f(\mathfrak{m})) \geq 1 \right\}$$
satisfy
\begin{align}\label{npx2A}
    |\mathcal{S}_\mfp(x)| 
    & = \frac{C_{\beta} x^\beta}{N(\mfp)} + \frac{C_{\mfp,\beta}' x^\beta}{N(\mfp)^{1+\eta}} + O_\beta \left( \frac{x^{\xi \beta}}{N(\mfp)^\xi}\right),
\end{align}
for some $\eta > 0$ and where the constant $C_{\mfp,\beta}'$ is uniformly bounded in $\mfp$.
Then for $\gamma \in \mathbb{R}$, we have
$$\lim_{x \rightarrow \infty} \frac{1}{|\mcs(x)|} \left| \left\{ \mfm \in \mcs(x) \ : \ N(\mfm) \geq 3, \ \frac{\omega(f(\mfm)) - \log \log N(\mfm)}{\sqrt{\log \log N(\mfm)}} \leq \gamma\right\} \right| = \Phi(\gamma),$$
where $\Phi(\gamma)$ is defined in \eqref{phi(a)}.
\end{thm}

As applications of the above theorem, we prove the Erd\H{o}s-Kac theorem over $h$-free and over $h$-full elements as the following:

\begin{thm}\label{erdoskacforomega}
Let $\mcp$, $\mcm$, and $X$ satisfy condition \eqref{star}. Let $x \in X$ and $h \geq 2$ be any integer. Let $\mathcal{\mcs}_h(x)$ denote the set of $h$-free elements in $\mcm$ with norm less than or equal to $x$. Then for $a \in \mathbb{R}$, we have
$$\lim_{x \rightarrow \infty} \frac{1}{|\mathcal{\mcs}_h(x)|} \bigg| \left\{ \mfm \in \mathcal{\mcs}_h(x) \ : N(\mfm) \geq 3, \ \ \frac{\omega(\mfm) - \log \log N(\mfm)}{\sqrt{\log \log N(\mfm)}} \leq a \right\} \bigg| = \Phi(a).$$
\end{thm}

\begin{thm}\label{erdoskacforomegahfull}
Let $\mcp$, $\mcm$, and $X$ satisfy condition \eqref{star}. Let $x \in X$ and $h \geq 2$ be any integer. Let $\mathcal{N}_h(x)$ denote the set of $h$-full elements in $\mcm$ with norm less than or equal to $x$. Then for $a \in \mathbb{R}$, we have
$$\lim_{x \rightarrow \infty} \frac{1}{|\mathcal{N}_h(x)|} \bigg| \left\{ \mfm \in \mathcal{N}_h(x) \ : \ N(\mfm) \geq 3, 
\ \frac{\omega(\mfm) - \log \log N(\mfm)}{\sqrt{\log \log N(\mfm)}} \leq a \right\} \bigg| = \Phi(a).$$
\end{thm}

Next, we show that \thmref{less-general} has more applications in the general context, by taking $f$ to be a non-identity function.

Let $k \geq 1$ be any integer. For each $\mfm \in \mcm$, we define the $\omega_k$-function as
$$\omega_k(\mfm) = \sum_{\substack{\mfp \in \mcp \\ n_\mfp(\mfm) = k}} 1,$$
the number of elements of $\mcp$ that generates $\mfm$, with multiplicity $k$.

For an element $\mfm \in \mcm$, let $\mfm_k$ be defined as
\begin{equation}\label{mk}
    \mfm_k = k \cdot \sum_{\substack{\mfp \\ n_\mfp(\mfm) = k}} \mfp.
\end{equation}
Thus
$$\mfm = \sum_{k \geq 1} \mfm_k.$$
We define the map $f_k : \mcs \rightarrow \mcm$ as
$$f_k(\mfm) = \mfm_k.$$
Note that, with the above definition, we have
$$\omega(f_k(\mfm)) = \omega(\mfm_k) = \sum_{\substack{\mfp \\ n_\mfp(\mfm_k) \geq 1}} 1 = \sum_{\substack{\mfp \\ n_\mfp(\mfm) = k}  } 1 = \omega_k(\mfm).$$

As more applications of \thmref{less-general}, we deduce that $\omega_1(\mfm)$ satisfies the Erd\H{o}s-Kac theorems over the subset of $h$-free elements, where $h \geq 2$. We prove:


\begin{thm}\label{erdoskacforomega1}
Let $\mcp$, $\mcm$, and $X$ satisfy condition \eqref{star}. Let $x \in X$ and $h \geq 2$ be any integer. Let $\mathcal{\mcs}_h(x)$ denote the set of $h$-free elements in $\mcm$ with norm less than or equal to $x$. Then for $a \in \mathbb{R}$, we have
$$\lim_{x \rightarrow \infty} \frac{1}{|\mathcal{\mcs}_h(x)|} \bigg| \left\{ \mfm \in \mathcal{\mcs}_h(x) \ : N(\mfm) \geq 3, \ \ \frac{\omega_1(\mfm) - \log \log N(\mfm)}{\sqrt{\log \log N(\mfm)}} \leq a \right\} \bigg| = \Phi(a).$$
\end{thm}

In the following, for convenience, we use $\mathcal{N}_1 = \mcm$ and $\mathcal{N}_1(x)$ to denote the set of elements in $\mcm$ with norm less than or equal to $x$. We call $\mathcal{N}_1$ the set of $1$-full elements. To encompass the set of $h$-full elements defined above and $\mcm$ under one notation, we use $\mathcal{N}_k$ where $k \geq 1$ to denote the set of $k$-full elements. We show that $\omega_k(\mfm)$ satisfies the Erd\H{o}s-Kac theorem over the subset of $k$-full elements in the following:

\begin{thm}\label{erdoskacforomegah}
Let $\mcp$, $\mcm$, and $X$ satisfy condition \eqref{star}. Let $x \in X$ and $k \geq 1$ be any integer. Let $\mathcal{N}_k(x)$ denote the set of $k$-full elements in $\mcm$ with norm less than or equal to $x$. Then for $a \in \mathbb{R}$, we have
$$\lim_{x \rightarrow \infty} \frac{1}{|\mathcal{N}_k(x)|} \bigg| \left\{ \mfm \in \mathcal{N}_k(x) \ : \ N(\mfm) \geq 3, 
\ \frac{\omega_k(\mfm) - \log \log N(\mfm)}{\sqrt{\log \log N(\mfm)}} \leq a \right\} \bigg| = \Phi(a).$$
\end{thm}

We extend our study to the generalized $\omega$-function related to a sequence, inspired by the work of Elma and Martin \cite{elmamartin} over natural numbers. Let $\mcA = (a_1,a_2,\ldots)$ be a sequence of complex numbers, and for some integer $k \geq 1$, the following property is satisfied:
\begin{equation}\label{Condition-ai}
  \sum_{\substack{\mfp \in \mcp \\ N(\mfp) \leq x^{1/k}}} \sum_{i \geq k+1} \frac{|a_i|}{N(\mfp)^{i/k}} = O(1).  
\end{equation}
The above relation produces a growth condition on $a_i$'s depending on the minimum value of the norm map. In particular, if we assume that $N(\mfp) \geq b$ for all $\mfp \in \mcp$, then a sufficient condition for \eqref{Condition-ai} to hold is

$$a_i \ll B^i \quad \text{ as } i \rightarrow \infty, \quad \text{ where } \quad 0 < B \leq b^{\frac{1}{k} - \alpha} \quad \text{ for some } \alpha > 0.$$

For this article, we shall assume $b = 2$, i.e., the minimum possible value of the norm of a prime element.

Let the generalized $\omega$-function, $\omega_\mcA$ attached to the sequence $\mcA$ be defined as

$$\omega_\mcA (\mfm) = \sum_{k \geq 1} a_k \ \omega(f_k(\mfm)) = \sum_{k \geq 1} a_k \ \omega_k(\mfm),$$
where the sum is finite for each $\mfm$. 

With the above definitions and restrictions, for an integer $k \geq 1$, we prove the following Erd\H{o}s-Kac theorem for $\omega_\mcA(\mfm)$ over $k$-full elements in $\mcm$.  
\begin{thm}\label{erdoskacforomegaA}
Let $\mcp$, $\mcm$, and $X$ satisfy condition \eqref{star}. Let $x \in X$. Let $\mcA = (a_1,a_2,\ldots)$ be a sequence of complex numbers. Let $k \in \mathbb{Z}_{>0}$ be such that $a_k \neq 0$ and the following property holds:
$$a_i \ll B^i \quad \text{ as } i \rightarrow \infty, \quad \text{ where } \quad 0 < B \leq 2^{\frac{1}{k} - \alpha} \quad \text{ for some } \alpha > 0.$$
Let $\mathcal{N}_k(x)$ denote the set of $k$-full elements in $\mcm$ with norm less than or equal to $x$. 
Then for $a \in \mathbb{R}$, we have
$$\lim_{x \rightarrow \infty} \frac{1}{|\mathcal{N}_k(x)|} \bigg| \left\{ \mfm \in \mathcal{N}_k(x) \ : \ N(\mfm) \geq 3, 
\ \frac{\frac{1}{a_k} \omega_\mcA(\mfm) - \log \log N(\mfm)}{\sqrt{\log \log N(\mfm)}} \leq a \right\} \bigg| = \Phi(a).$$
\end{thm}

In the next result, we state the Erd\H{o}s-Kac theorem for $\frac{1}{a_1} \omega_\mcA(\mfm)$ over $h$-free elements. The set of $h$-free elements in $\mcm$ has a positive density and thus the proof of the theorem follows similarly to the case of $1$-full elements in \thmref{erdoskacforomegaA}. To avoid repetition, we don't provide the proof in this article.

\begin{thm}\label{erdoskacforomegaA-hfree}
Let $\mcp$, $\mcm$, and $X$ satisfy condition \eqref{star}. Let $x \in X$. Let $h \geq 2$ be any integer and let $\mathcal{S}_h(x)$ denote the set of $h$-free elements in $\mcm$ with norm less than or equal to $x$. Let $\mcA = (a_1,a_2,\ldots)$ be a sequence of complex numbers, satisfying $a_1 \neq 0$ and
$$a_i \ll B^i \quad \text{ as } i \rightarrow \infty, \quad \text{ where } \quad 0 < B \leq 2^{1 - \alpha} \quad \text{ for some } \alpha > 0.$$
Then for $a \in \mathbb{R}$, we have
$$\lim_{x \rightarrow \infty} \frac{1}{|\mathcal{S}_h(x)|} \bigg| \left\{ \mfm \in \mathcal{S}_h(x) \ : \ N(\mfm) \geq 3, 
\ \frac{\frac{1}{a_1} \omega_\mcA(\mfm) - \log \log N(\mfm)}{\sqrt{\log \log N(\mfm)}} \leq a \right\} \bigg| = \Phi(a).$$
\end{thm}


We provide applications of \thmref{erdoskacforomegaA} to other well-known functions, in particular, the prime counting $\Omega$-function and the divisor counting function. Note that our method provides a novel approach to study the prime divisor counting functions, i.e., the $\omega$-function and the $\Omega$-function, and the divisor counting function in a single framework. As these functions are of prime importance to the various mathematicians, our work contributes significantly to the literature.

For each $\mfm \in \mcm$, we define
$$\Omega(\mfm) = \sum_{\substack{\mfp \in \mcp \\ n_\mfp(\mfm) \geq 1}} n_\mfp(\mfm),$$
the number of elements of $\mcp$ that generates $\mfm$, counted with multiplicity. We say $\mfd \in \mcm$ is a divisor of $\mfm$ and denote it as $\mfd | \mfm$ if $n_\mfp(\mfd) \leq n_\mfp(\mfm)$ for all $\mfp \in \mcp$. We define
$$d(\mfm) = \sum_{\substack{\mfd \in \mfm \\ \mfd | \mfm}} 1,$$
the number of divisors of $\mfm$. When $\mcm = \mathbb{Z}_{>0}$, $\Omega(n)$ counts the total number of primes diving the natural number $n$, and $d(n)$ counts the total number of divisors of $n$.

Notice that
$$\Omega(\mfm) = \sum_{k \geq 1} k \left(\sum_{\substack{\mfp \in \mcp \\ n_\mfp(\mfm) = k}} 1 \right) = \sum_{k \geq 1} k \cdot \omega_k(\mfm).$$
For $\mcA = (1,2,3,\ldots,n,\ldots)$ and for any integer $k \geq 1$, we have $a_i = i \ll 2^{i/(2k)}$ as $i \rightarrow \infty$, satisfying the hypothesis of \thmref{erdoskacforomegaA}. Thus, we obtain the following corollary, called the Erd\H{o}s-Kac theorem for $\Omega(\mfm)$ over $k$-full elements:

\begin{cor}
Let $\mcp$, $\mcm$, and $X$ satisfy condition \eqref{star}. Let $x \in X$ and $k \geq 1$ be any integer. Let $\mathcal{N}_k(x)$ denote the set of $k$-full elements in $\mcm$ with norm less than or equal to $x$. Then for $a \in \mathbb{R}$, we have
$$\lim_{x \rightarrow \infty} \frac{1}{|\mathcal{N}_k(x)|} \bigg| \left\{ \mfm \in \mathcal{N}_k(x) \ : \ N(\mfm) \geq 3, 
\ \frac{\frac{1}{k} \Omega(\mfm) - \log \log N(\mfm)}{\sqrt{\log \log N(\mfm)}} \leq a \right\} \bigg| = \Phi(a).$$
\end{cor}

For the $d(\mfm)$ function, one can easily deduce that
$$d(\mfm) = \sum_{\substack{\mfd \in \mfm \\ \mfd | \mfm}} 1 = \prod_{\substack{\mfp \\ n_\mfp(\mfm) \geq 1}} (n_\mfp(\mfm) + 1).$$
Thus,
$$\log d(\mfm) = \sum\limits_{\substack{\mfp \\ n_\mfp(\mfm) \geq 1}} \log (n_\mfp(\mfm) + 1) = \sum_{k \geq 1} \log(k+1) \left(\sum\limits_{\substack{\mfp \\ n_\mfp(\mfm) = k}} 1 \right) = \sum_{k \geq 1} \log(k+1) \cdot \omega_k(\mfm).$$

Taking $\mcA = (\log 2, \log 3, \log 4, \ldots, \log(n+1), \ldots)$ and for any integer $k \geq 1$, we have $a_i = \log(i+1) \ll 2^{i/(2k)}$ as $i \rightarrow \infty$, satisfying the hypothesis of \thmref{erdoskacforomegaA}. Thus, we obtain the following corollary, called the Erd\H{o}s-Kac theorem for $\log d(\mfm)$ over $k$-full elements:

\begin{cor}
    Let $\mcp$, $\mcm$, and $X$ satisfy condition \eqref{star}. Let $x \in X$ and $k \geq 1$ be any integer. Let $\mathcal{N}_k(x)$ denote the set of $k$-full elements in $\mcm$ with norm less than or equal to $x$. Then for $a \in \mathbb{R}$, we have
$$\lim_{x \rightarrow \infty} \frac{1}{|\mathcal{N}_k(x)|} \bigg| \left\{ \mfm \in \mathcal{N}_k(x) \ : \ N(\mfm) \geq 3, 
\ \frac{\frac{1}{\log (k+1)} \log d(\mfm) - \log \log N(\mfm)}{\sqrt{\log \log N(\mfm)}} \leq a \right\} \bigg| = \Phi(a).$$
\end{cor}

Next, we show that \thmref{erdoskacforomegaA} can also be applied to some new functions. 
Let $\omega_T(\mfm)$ denote the difference in the number of prime elements in the factorization of $\mfm$ with odd multiplicity and the number of prime elements in the factorization of $\mfm$ with even multiplicity, i.e.,
\begin{equation}\label{def-omegaT}
    \omega_T(\mfm) = \sum_{\substack{\mfp \\ n_\mfp(\mfm) \ \text{odd}}} 1 - \sum_{\substack{\mfp \\ n_\mfp(\mfm) \ \text{even}}} 1 = \sum_{\mfp} (-1)^{n_\mfp(\mfm) - 1} = \sum_{k \geq 1} (-1)^{k-1} \omega_k(\mfm).
\end{equation}
Taking $\mcA = (1, -1, 1, -1, \ldots, (-1)^{n-1}, \ldots)$ and for any integer $k \geq 1$, we have $(-1)^{i-1} \ll 2^{i/(2k)}$ as $i \rightarrow \infty$. Thus, as another application of \thmref{erdoskacforomegaA}, we obtain

\begin{cor}
    Let $\mcp$, $\mcm$, and $X$ satisfy condition \eqref{star}. Let $x \in X$ and $k \geq 1$ be any integer. Let $\mathcal{N}_k(x)$ denote the set of $k$-full elements in $\mcm$ with norm less than or equal to $x$. Then for $a \in \mathbb{R}$, we have
$$\lim_{x \rightarrow \infty} \frac{1}{|\mathcal{N}_k(x)|} \bigg| \left\{ \mfm \in \mathcal{N}_k(x) \ : \ N(\mfm) \geq 3, 
\ \frac{(-1)^{k-1} \omega_T(\mfm) - \log \log N(\mfm)}{\sqrt{\log \log N(\mfm)}} \leq a \right\} \bigg| = \Phi(a).$$
\end{cor}

Finally, in Section \ref{applications}, we provide various applications of our general setting. Let $\mcA = (a_1, a_2, \ldots)$ be a sequence of complex numbers satisfying any of the following types:
\begin{enumerate}
    \item[(1)] if $a_i = 1$ for all $i \in \mathbb{Z}_{>0}$, i.e., $\omega_\mcA(\mfm) = \omega(\mfm)$,
    \item[(2)] if $a_i = i$ for all $i \in \mathbb{Z}_{>0}$, i.e., $\omega_\mcA(\mfm) = \Omega(\mfm)$,
    \item[(3)] if $a_i = \log(i+1)$ for all $i \in \mathbb{Z}_{>0}$, i.e., $\omega_\mcA(\mfm) = \log d(\mfm)$,
    \item[(4)] if $a_i = 1$ for all odd $i$ and $a_i = - 1$ for all even $i$, i.e., $\omega_\mcA(\mfm) = \omega_T(\mfm)$.
    \item[(5)] if $a_i = 0$ for all $i \neq k$ and $a_k = 1$, i.e., $\omega_\mcA(\mfm) = \omega_k(\mfm)$.
\end{enumerate}

For such $\mcA'$s, we apply \thmref{erdoskacforomegaA} and \thmref{erdoskacforomegaA-hfree} to prove the Erd\H{o}s-Kac theorems for the $\omega_\mcA$-function in number fields, global function fields, and geometrically irreducible projective varieties, demonstrating the broad applicability of our approach.

\section{Review of Probability Theory}
In this section, we review some results from probability theory that are essential for our study. We repeat \cite[Section 2]{dkl3} here for the easiness of the readers. Interested readers can find a more detailed version of the results mentioned in this section in \cite[Section 2]{liu}.

Let $X$ be a random variable with a probability measure P. For a real number $t$, let $F(t)$ be the distribution function of $X$ defined as 
$$F(t) := P(X \leq t).$$
The expectation of $X$ is defined as
$$\textnormal{E}(X) := \int_{-\infty}^\infty t \ d F(t).$$
The variance of $X$, denoted as Var($X$), which measures the deviation of $X$ from its expectation is defined as
$$\textnormal{Var}(X) := \textnormal{E}(X^2) - (\textnormal{E}(X))^2.$$
Moreover, if $Y$ is another random variable with the same probability measure P, we have
$$\textnormal{E}(X+Y) = \textnormal{E}(X) + \textnormal{E}(Y).$$
The above property is called the linearity of expectation. Additionally, if $X$ and $Y$ are independent, i.e., for all $x \in X$ and for all $y \in Y$, 
$$P(X \leq x, Y \leq y) = \textnormal{P}(X \leq x) \cdot \textnormal{P}(Y \leq y),$$
then we have
$$\textnormal{E}(X \cdot Y) = \textnormal{E}(X) \cdot \textnormal{E}(Y),$$
and
$$\textnormal{Var}(X+Y) = \textnormal{Var}(X) + \textnormal{Var}(Y).$$
Given a sequence of random variables $\{X_n\}$ and $\alpha \in \mathbb{R}$, we say $\{X_n\}$ \textit{converges in probability} to $\alpha$ if for any $\epsilon > 0$,
$$\lim_{n \rightarrow \infty} P(|X_n - \alpha| > \epsilon) = 0.$$
We denote this by
$$X_n \xlongrightarrow{P}\alpha.$$
Using the above definitions, we list the following facts from probability theory as mentioned in the third author's work \cite[Page 595-596]{liu}.
\begin{fact}\label{fact1}
    Given a sequence of random variables $\{X_n\}$, if 
    $$\lim_{n \rightarrow \infty} \textnormal{E} (|X_n|) = 0,$$
    we have
    $$X_n \xlongrightarrow{P} 0.$$
\end{fact}
\begin{fact}\label{fact2}
    Let $\{ X_n \}$, $\{ Y_n \}$, and $\{ U_n \}$ be sequences of random variables with the same probability measure P. Let $U$ be a distribution function. Suppose
    $$X_n \xlongrightarrow{P} 1 \quad \textnormal{ and } \quad Y_n \xlongrightarrow{P} 0.$$
    For all $\gamma \in \mathbb{R}$, we have
    $$\lim_{n \rightarrow \infty} P(U_n \leq \gamma) = U(\gamma)$$
    if and only if
    $$\lim_{n \rightarrow \infty} P\left( (X_n U_n + Y_n) \leq \gamma \right) = U(\gamma).$$
\end{fact}
Let $\Phi(\gamma)$ denote the Gaussian normal distribution as defined in \eqref{phi(a)}. For $r \in \mathbb{Z}_{>0}$, the $r$-th moment of $\Phi$ is defined as
$$\mu_r := \int_{-\infty}^\infty t^r d \Phi(t).$$
Then we have:
\begin{fact}\label{fact3}
    Given a sequence of distribution functions $\{F_n\}$, if for all $r \in \mathbb{Z}_{>0}$,
    $$\lim_{n \rightarrow \infty} \int_{-\infty}^\infty t^r d F_n(t) = \mu_r,$$
    then for all $\gamma \in \mathbb{R}$, we have
    $$\lim_{n \rightarrow \infty} F_n(\gamma) = \Phi(\gamma).$$
\end{fact}
As a converse of the above fact, we have
\begin{fact}\label{fact4}
    Let $r \in \mathbb{Z}_{>0}$. Given a sequence of distribution functions $\{F_n\}$, if 
    $$\lim_{n \rightarrow \infty} F_n(\gamma) = \Phi(\gamma), \quad \textnormal{for all } \gamma \in \mathbb{R}$$
    and
    $$\sup_n \left\{ \int_{-\infty}^\infty |t|^{r+\delta} dF_n(t) \right\} < \infty, \quad \textnormal{for some } \delta = \delta(r) > 0,$$
    we have
    $$\lim_{n \rightarrow \infty} \int_{-\infty}^\infty t^r d F_n(t) = \mu_r.$$
\end{fact}
The next fact is a special case of the Central Limit Theorem.
\begin{fact}\label{fact5}
    Let $X_1,X_2,\ldots, X_i, \ldots$ be a sequence of independent random variables and $\textnormal{Im}(X_i)$ is the image of $X_i$. Suppose
    \begin{enumerate}
        \item $\sup_{i} \{ \textnormal{Im}(X_i) \} < \infty$,
        \item $\textnormal{E}(X_i) = 0$ and $\textnormal{Var}(X_i) < \infty$ for all $i$.
    \end{enumerate}
    For $n \in \mathbb{Z}_{>0}$, let $\Phi_n$ be the normalization of $X_1,X_2,\ldots,X_n$ defined as
    $$\Phi_n := \left( \sum_{i=1}^{n} X_i \right) \bigg\slash \left( \sum_{i=1}^n \textnormal{Var}(X_i) \right)^{1/2}.$$
    If $\sum_{i=1}^\infty \textnormal{Var}(X_i)$ diverges, then we have
    $$\lim_{n \rightarrow \infty} P(\Phi_n \leq \gamma) = \Phi(\gamma).$$
\end{fact}
\section{Essential lemmas}
In this section, we list all the lemmas required to prove our theorems. The first three lemmas establish statements equivalent to \thmref{yurugen}, and thus proving any equivalent statement would be sufficient in proving the theorem. The next two lemmas establish results necessary to prove one of the equivalent conditions mentioned in the third lemma of this section. Together, these lemmas prove \thmref{yurugen} in the next section. These results bear a close resemblance to the results from \cite[Section 3]{dkl3}. However, because of subtle changes to the arguments that involve a new function $f$, we present the results in detail here.

The final set of three lemmas in this section establishes results involving prime elements required to complete the proofs of all other theorems mentioned in Section \ref{intro}.

Let $\mcp$, $\mcm$, $\mcs$, $X$, and $f$ be defined as in Section 1 and assume that they satisfy \eqref{star}, \eqref{Scondition} and the conditions (a) to (f). For $\mfm \in \mcs$ and $x \in X$, we define
$$P_{\mcs,x} \{ \mfm \ : \ \mfm \text{ satisfies some conditions} \}$$
to be the quantity 
$$\frac{1}{|\mcs(x)|} \left| \{ \mfm \in \mcs(x) \ : \  \mfm \text{ satisfies some conditions} \} \right|.$$
Note that $P_{\mcs,x}$ is a probability measure on $\mcs$. Let $g$ be a function from $\mcs$ to $\mathbb{R}$. The expectation of $g$ with respect to $P_{\mcs,x}$ is denoted by
$$\textnormal{E}_{\mcs,x} \{ \mfm : g(\mfm) \} := \frac{1}{|\mcs(x)|} \sum_{\mfm \in \mcs(x)} g(\mfm).$$

The first lemma gives an equivalent statement of \thmref{yurugen}.
\begin{lma}\label{linkN(m)x}
    $$\lim_{x \rightarrow \infty} P_{\mcs,x} \bigg\{ \mfm \ : \ N(\mfm) \geq 3, \ \frac{\omega(f(\mfm)) - \log \log N(\mfm)}{\sqrt{\log \log N(\mfm)}} \leq \gamma \bigg\} = \Phi(\gamma)$$
    if and only if
    $$\lim_{x \rightarrow \infty} P_{\mcs,x} \bigg\{ \mfm \ : \ \frac{\omega(f(\mfm)) - \log \log x}{\sqrt{\log \log x}} \leq \gamma \bigg\} = \Phi(\gamma).$$
\end{lma}
\begin{proof}
The proof closely follows the steps of the proof of \cite[Lemma 3]{liu}. First note that
\begin{align*}
    \frac{\omega(f(\mfm)) - \log \log x}{\sqrt{\log \log x}} & = \frac{\omega(f(\mfm)) - \log \log N(\mfm)}{\sqrt{\log \log N(\mfm)}} \frac{\sqrt{\log \log N(\mfm)}}{\sqrt{\log \log x}} \\
    & \hspace{.5cm} + \frac{\log \log N(\mfm) - \log \log x}{\sqrt{\log \log x}}.
\end{align*}
Thus by Fact 2 and our assumption that $\mathcal{S}$ is infinite, to prove the lemma, it suffices to show that for any $\epsilon >0$, 
$$\lim_{x \rightarrow \infty} P_{\mathcal{S},x} \bigg\{ \mathfrak{m} \ : \ N(\mathfrak{m}) \geq 3, \ \bigg| \frac{\sqrt{\log \log N(\mathfrak{m})}}{\sqrt{\log \log x}} - 1 \bigg| > \epsilon \bigg\} = 0 $$
and
$$\lim_{x \rightarrow \infty} P_{\mathcal{S},x} \bigg\{ \mathfrak{m} \ : \ N(\mathfrak{m}) \geq 3, \ \bigg| \frac{\log \log N(\mathfrak{m}) - \log \log x}{\sqrt{\log \log x}} \bigg| > \epsilon \bigg\} = 0.$$
We show this by repeating the steps in \cite[Lemma 3.1, Page 10]{dkl3}, and thus complete the proof.
\end{proof}

Let $\beta$ be a constant with $0 < \beta \leq 1$ and $y = y(x) < x^\beta$ satisfying the conditions (a)-(f) as mentioned in Section \ref{intro}. For $\mfm \in \mcm$, we define the truncated function
$$\omega_y(f(\mfm)) = \left| \{ \ell \in \mcp \ : \ N(\ell) \leq y, \ n_\ell(f(\mfm)) \geq 1 \} \right|.$$

The next result establishes another equivalent formulation of the Erd\H{o}s-Kac theorem in terms of $\omega_y$. 

\begin{lma}\label{linkomegay}
    $$\lim_{x \rightarrow \infty} P_{\mcs,x} \bigg\{ \mfm \ : \ \frac{\omega(f(\mfm)) - \log \log x}{\sqrt{\log \log x}} \leq \gamma \bigg\} = \Phi(\gamma)$$
    if and only if
    $$\lim_{x \rightarrow \infty} P_{\mcs,x} \bigg\{ \mfm \ : \ \frac{\omega_y(f(\mfm)) - \log \log x}{\sqrt{\log \log x}} \leq \gamma \bigg\} = \Phi(\gamma).$$
\end{lma}
\begin{proof}
    Note that
    $$\frac{\omega_y(f(\mfm)) - \log \log x}{\sqrt{\log \log x}} = \frac{\omega(f(\mfm)) - \log \log x}{\sqrt{\log \log x}} + \frac{\omega_y(f(\mfm)) - \omega(f(\mfm))}{\sqrt{\log \log x}}.$$
    Thus, by Fact 1 and Fact 2, to prove the lemma, it suffices to prove
    $$\lim_{x \rightarrow \infty} \textnormal{E}_{\mcs,x} \bigg\{ \mfm \ : \ \bigg| \frac{\omega(f(\mfm)) - \omega_y(f(\mfm))}{\sqrt{\log \log x}} \bigg| \bigg\} = 0.$$ 
    Notice that
    \begin{align*}
        & \sum_{\substack{\mfm \in \mcs \\ N(\mfm) \leq x}}  |\omega(f(\mfm)) - \omega_y(f(\mfm))| \\
        & = \sum_{\substack{\mfm \in \mcs \\ N(\mfm) \leq x}}  \sum_{\substack{\ell \in \mcp \\ N(\ell) > y, \ n_\ell(f(\mfm)) \geq 1 }} 1  \\
        & = \sum_{\substack{\ell \in \mcp \\ y < N(\ell) \leq x^\beta }} \sum_{\substack{\mfm \in \mcs \\ N(\mfm) \leq x, \ n_\ell(f(\mfm)) \geq 1}} 1 + \sum_{\substack{\mfm \in \mcs \\ N(\mfm) \leq x}}  \sum_{\substack{\ell \in \mcp \\ N(\ell) > x^\beta, \ n_\ell(f(\mfm)) \geq 1 }} 1.
    \end{align*}
Using the definition of $\lambda_\ell$ and $e_\ell$, and the conditions (a), (b), and (c), we obtain
\begin{align*}
    \sum_{\substack{\mfm \in \mcs \\ N(\mfm) \leq x}}  |\omega(f(\mfm)) - \omega_y(f(\mfm))| & = \sum_{\substack{\ell \in \mcp \\ y < N(\ell) \leq x^\beta }} |\mcs(x)|(\lambda_\ell+ e_\ell) + O(|\mcs(x)|) \\
    & = o(|\mcs(x)|(\log \log x)^{1/2}) + O(|\mcs(x)|).
\end{align*}
Thus, we have
$$\textnormal{E}_{\mcs,x} \bigg\{ m \ : \ \bigg| \frac{\omega(f(\mfm)) - \omega_y(f(\mfm))}{\sqrt{\log \log x}} \bigg| \bigg\} = \frac{o(|\mcs(x)|(\log \log x)^{1/2})}{|\mcs(x)|(\log \log x)^{1/2}} = o(1),$$
which completes the proof.
\end{proof}
For $\ell \in \mcp$, we define the independent random variable $X_\ell$ by
$$P(X_\ell =1) = \lambda_\ell$$
and
$$P(X_\ell=0) = 1 - \lambda_\ell.$$
We define a new random variable $\mcs_y$ by
$$\mcs_y := \sum_{\substack{\ell \in \mcp \\ N(\ell) \leq y}} X_\ell.$$
Note that, by conditions (d) and (e), we have the expectation and variance of the random variable $\mcs_y$ as
$$\text{E}(\mcs_y) = \sum_{N(\ell) \leq y} \lambda_\ell = \log \log x + o \left( (\log \log x)^{1/2} \right),$$
and
$$\text{Var}(\mcs_y) = \sum_{N(\ell) \leq y} \lambda_\ell(1 - \lambda_\ell) = \log \log x + o \left( (\log \log x)^{1/2} \right).$$
Note that, we will use the notation $\textnormal{E}(\cdot)$ and $\textnormal{E}_{\mcs,x}\{\cdot\}$ respectively to distinguish the expectation of a random variable from the expectation of a function with respect to $P_{\mcs,x}$. However, in most cases, they will represent the same values.

The above setup leads us to another reformulation of \thmref{yurugen} in terms of $\textnormal{E}(\mcs_y)$.
\begin{lma}\label{linkomegayESy}
$$\lim_{x \rightarrow \infty} P_{\mcs,x} \bigg\{ \mfm \ : \ \frac{\omega_y(f(\mfm)) - \log \log x}{\sqrt{\log \log x}} \leq \gamma \bigg\} = \Phi(\gamma)$$
if and only if
$$\lim_{x \rightarrow \infty} P_{\mcs,x} \bigg\{ \mfm \ : \ \frac{\omega_y(f(\mfm)) - \textnormal{E}(\mcs_y)}{\sqrt{\textnormal{Var}(\mcs_y)}} \leq \gamma \bigg\} = \Phi(\gamma).$$
\end{lma}
\begin{proof}
Note that
$$\frac{\omega_y(f(\mfm)) - \textnormal{E}(\mcs_y)}{\sqrt{\textnormal{Var}(\mcs_y)}} = \frac{\omega_y(f(\mfm)) - \log \log x}{\sqrt{\log \log x}} \frac{\sqrt{\log \log x}}{\sqrt{\textnormal{Var}(\mcs_y)}} + \frac{\log \log x - \textnormal{E}(\mcs_y)}{\sqrt{\textnormal{Var}(\mcs_y)}}.$$
Since
$$\text{Var}(\mcs_y) = \sum_{N(\ell) \leq y} \lambda_\ell(1 - \lambda_\ell) = \log \log x + o \left( (\log \log x)^{1/2} \right),$$
we have
$$\frac{\sqrt{\log \log x}}{\sqrt{\textnormal{Var}(\mcs_y)}} \xlongrightarrow{P} 1,$$
where $\xlongrightarrow{P}$ denotes the convergence in probability. Moreover, since 
$$\text{E}(\mcs_y) = \sum_{N(\ell) \leq y} \lambda_\ell = \log \log x + o \left( (\log \log x)^{1/2} \right),$$
we obtain
$$\lim_{x \rightarrow \infty} \textnormal{E}_{\mcs,x} \bigg\{ \mfm \ : \  \bigg| \frac{\textnormal{E}(\mcs_y) - \log \log x}{\sqrt{\textnormal{Var}(\mcs_y)}} \bigg| \bigg\} = 0.$$
Finally, by using Fact 1 and Fact 2, we complete the proof of the equivalence mentioned in the lemma.
\end{proof}
Next, we introduce another set of random variables. For $\ell \in \mcp$, we define a  random variable $\delta_\ell:\mcm \rightarrow \mathbb{R}$ by
$$\delta_\ell(\mfm) := \begin{cases}
    1 & \text{if } n_\ell(\mfm) \geq 1, \\
    0 & \text{otherwise.}
\end{cases}$$
Thus, we can write
$$\omega_y(f(\mfm)) = \sum_{\substack{\ell \in \mcp \\ N(\ell) \leq y, \ n_\ell(f(\mfm)) \geq 1 }} 1 = \sum_{\substack{\ell \in \mcp \\ N(\ell) \leq y}} \delta_\ell(f(\mfm)).$$
Notice that for a fixed $\ell \in \mcp$ and $x \in X$, by definition, we have
$$P_{\mcs,x} \{ m \ : \ \delta_\ell(f(\mfm)) = 1 \} = \lambda_\ell + e_\ell.$$
Since the expectations of random variables $X_\ell$ and $\delta_\ell$ are close, the sum $\mcs_y$ is a good approximation of $\omega_y$. Indeed, the $r$-th moments of their normalizations are equal as $x \rightarrow \infty$, which we prove in the following result.
\begin{lma}\label{rthmoment}
 Let $r \in \mathbb{Z}_{>0}$. We have
 $$\lim_{x \rightarrow \infty} \bigg| \textnormal{E}_{\mcs,x} \bigg\{ \bigg( \frac{\omega_y(f(\mfm)) - \textnormal{E}(\mcs_y)}{\sqrt{\textnormal{Var}(\mcs_y)}} \bigg)^r \bigg\}- E\bigg( \bigg( \frac{\mcs_y - \textnormal{E}(\mcs_y)}{\sqrt{\textnormal{Var}(\mcs_y)}} \bigg)^r\bigg) \bigg| = 0.$$
\end{lma}
\begin{proof}
The proof follows from repeating the exact steps of the proof of \cite[Lemma 3.4]{dkl3} with $\omega_y(\mfm)$ replaced with $\omega_y(f(\mfm))$.
\end{proof}
The next result is about the $r$-th moment of the random variable $\mcs_y$.
\begin{lma}\label{finite_r-thmoment}
    For $r \in \mathbb{Z}_{>0}$,
    $$\sup_y \bigg| \textnormal{E} \bigg( \bigg(\frac{\mcs_y - \textnormal{E}(\mcs_y)}{\sqrt{\textnormal{Var}(\mcs_y)}} \bigg)^r \bigg) \bigg| < \infty.$$
\end{lma}
\begin{proof}
The proof follows from repeating the exact steps of the proof of \cite[Lemma 3.5]{dkl3}.
\end{proof}

Next, we recall the following results regarding sums over prime elements necessary for our study:

\begin{lma}\label{boundnm}\cite[Lemma 2.2]{dkl5}
   Let $\mathcal{P}, \mcm$, and $X$ satisfy the condition \eqref{star}. Let $x \in X$ and $\alpha$ be a real number. We have
\begin{enumerate}
\item If $0 \leq \alpha < 1$, 
       $$\sum_{\substack{\mfp \in \mathcal{P} \\ N(\mfp) \leq x}} \frac{1}{N(\mfp)^\alpha} = O_\alpha \left( \frac{x^{1- \alpha}}{\log x} \right).$$
\item If $\alpha > 1$, then
       $$\sum_{\substack{\mfp \in \mathcal{P} \\ N(\mfp) \geq x}} \frac{1}{N(\mfp)^\alpha} = O \left( \frac{1}{(\alpha-1)x^{\alpha-1} (\log x)} \right).$$
       \item If $\alpha > 1$, then
       $$\sum_{\substack{\mfp \in \mathcal{P} \\ N(\mfp) \leq x}} \frac{1}{N(\mfp)^\alpha} = O_\alpha(1).$$
       \item As a generalization of Mertens' theorem, we have
       $$\sum_{\substack{\mfp \in \mathcal{P} \\ N(\mfp) \leq x}} \frac{1}{N(\mfp)} = \log \log x + \mfa + O \left( \frac{1}{\log x} \right),$$
       where $\mfa$ some constant that depends only on $\mcp$.
\end{enumerate}
\end{lma}

Finally, we recall the following results regarding the density of particular sequences of $h$-free and $h$-full elements in $\mcm$.
\begin{lma}\label{hfreeidealrestrict}\cite[Lemma 3.1]{dkl5}
    Let $\mathcal{P}, \mcm$, and $X$ satisfy the condition \eqref{star}. Let  $x \in X$, $h \geq 2$ and $r \geq 1$ be integers. Let $\ell_1,\ldots,\ell_r$ be fixed distinct prime elements and $\mathcal{S}_{h,\ell_1,\ldots,\ell_r}(x)$ denote the set of $h$-free elements $\mfm \in \mcm$ with norm $N(\mfm) \leq x$ and with $n_{\ell_i}(\mfm) = 0$ for all $i \in \{1,\cdots, r\}$. Then, we have
    $$|\mathcal{S}_{h,\ell_1,\ldots,\ell_r}(x)| = \prod_{i=1}^r \left( \frac{N(\ell_i)^h - N(\ell_i)^{h-1}}{N(\ell_i)^h - 1} \right) \frac{\kappa}{\zeta_\mcm(h)} x + O_{h,r} \left( R_{\mathcal{S}_h}(x) \right),$$
    where $R_{\mathcal{S}_h}(x)$ is defined in \eqref{RSh(x)}.
\end{lma}
\begin{lma}\label{hfullidealsrestrict}\cite[Lemma 4.2]{dkl5}
Let $\mathcal{P}, \mcm$, and $X$ satisfy the condition \eqref{star}. Let  $x \in X$, $h \geq 2$ and $r \geq 1$ be integers. Let $\ell_1, \cdots, \ell_r$ be fixed distinct prime elements and $\mathcal{N}_{h,\ell_1,\cdots, \ell_r}(x)$ denote the set of $h$-full elements $\mfm \in \mcm$ with norm $N(\mfm) \leq x$ and with $n_{\ell_i}(\mfm) = 0$ for all $i \in \{1,\cdots, r\}$. Then, we have
$$|\mathcal{N}_{h,\ell_1,\cdots, \ell_r}(x)| = \prod_{i=1}^r \frac{\kappa \gamma_{\scaleto{h}{4.5pt}}}{\left( 1+ \frac{N(\ell_i)^{-1}}{1-N(\ell_i)^{-1/h}} \right)} x^{1/h} + O_{h,r} \big( R_{\mathcal{N}_h}(x) \big),$$
where $\gamma_{\scaleto{h}{4.5pt}}$ is defined in \eqref{gammahk} and where $R_{\mathcal{N}_h}(x)$ is defined in \eqref{E2(x)}.
\end{lma}

\section{The Erd\H{o}s-Kac theorem over subsets}
In this section, we prove the Erd\H{o}s-Kac theorem over any subset of any abelian monoid satisfying the set of conditions mentioned in \thmref{yurugen}.
\begin{proof}[\textbf{Proof of \thmref{yurugen}}]
Given $\mcp, \mcm, X, \mcs$, and $f$ as in the statement of the theorem, suppose for all $x \in X$, there exists a constant $\beta$ with $0 < \beta \leq 1$ and $y = y(x) < x^\beta$ such that the conditions \eqref{Scondition} and (a) to (f) satisfy. For $\mfm \in \mcs$, we want to show the quantity 
$$\frac{\omega(f(\mfm)) - \log \log N(\mfm)}{\sqrt{\log \log N(\mfm)}}$$
satisfies the normal distribution. By the equivalent statements in \lmaref{linkN(m)x},  \lmaref{linkomegay}, and \lmaref{linkomegayESy}, to prove \thmref{yurugen}, it suffices to prove
$$\lim_{x \rightarrow \infty} P_{\mcs,x} \bigg\{ \mfm \ : \ \frac{\omega_y(f(\mfm)) - \textnormal{E}(\mcs_y)}{\sqrt{\textnormal{Var}(\mcs_y)}} \leq \gamma \bigg\} = \Phi(\gamma).$$
The distribution function $F_y$ respect to $P_{\mcs,x}$ is defined by
$$F_y(\gamma) := P_{\mcs,x} \bigg\{ \mfm  \ : \ \frac{\omega_y(f(\mfm)) - \textnormal{E}(\mcs_y)}{\sqrt{\textnormal{Var}(\mcs_y)}} \leq \gamma \bigg\}.$$
Notice that the $r$-th moment of $F_y$ can be written as
\begin{align*}
    & \int_{-\infty}^\infty t^r dF_y(t) \\
    & = \sum_{t = -\infty}^\infty \bigg\{ \lim_{u \rightarrow \infty} \sum_{i=1}^u (t+i/u)^r \bigg( F_y(t+i/u) - F_y(t+(i-1)/u) \bigg) \bigg\} \\
    & = \sum_{t = -\infty}^\infty \bigg\{ \lim_{u \rightarrow \infty} \sum_{i=1}^u (t+i/u)^r P_{\mcs,x} \bigg\{ \mfm \ : \ (t+(i-1)/u) < \frac{\omega_y(f(\mfm)) - \textnormal{E}(\mcs_y)}{\sqrt{\textnormal{Var}(\mcs_y)}} \leq (t + i/u) \bigg\} \bigg\}. 
\end{align*}
Thus, by the definition of $P_{\mcs,x}$, we have
$$\int_{-\infty}^\infty t^r dF_y(t) = \frac{1}{|\mcs(x)|} \sum_{\mfm \in \mcs(x)} \left( \frac{\omega_y(f(\mfm)) - \textnormal{E}(\mcs_y)}{\sqrt{\textnormal{Var}(\mcs_y)}} \right)^r = \textnormal{E}_{\mcs,x} \bigg\{ \mfm : \bigg( \frac{\omega_y(f(\mfm)) - \textnormal{E}(\mcs_y)}{\sqrt{\textnormal{Var}(\mcs_y)}} \bigg)^r \bigg\}.$$
Hence, to prove
$$\lim_{x \rightarrow \infty} F_y(\gamma) = \Phi(\gamma),$$
by Fact 3, it suffices to show that for all $r \in \mathbb{Z}_{>0}$,
$$\lim_{x \rightarrow \infty} \textnormal{E}_{\mcs,x} \bigg\{ \mfm : \bigg( \frac{\omega_y(f(\mfm)) - \textnormal{E}(\mcs_y)}{\sqrt{\textnormal{Var}(\mcs_y)}} \bigg)^r \bigg\} = \mu_r.$$
By \lmaref{rthmoment}, we observe that the last equality holds if
$$\lim_{x \rightarrow \infty} \textnormal{E} \bigg( \bigg( \frac{\mcs_y(f(\mfm)) - \textnormal{E}(\mcs_y)}{\sqrt{\textnormal{Var}(\mcs_y)}} \bigg)^r \bigg) = \mu_r.$$
We define a new random variable $\Phi$ by
$$\Phi_y := \frac{\mcs_y - \textnormal{E}(\mcs_y)}{\sqrt{\textnormal{Var}(\mcs_y)}}.$$
Note that \lmaref{finite_r-thmoment} ensures that any sequence of $\Phi_y$'s satisfies the hypothesis of Fact 5. Thus, by the Central Limit theorem given in Fact 5, we have
$$\lim_{x \rightarrow \infty} P(\Phi_{y} \leq \gamma) = \Phi(\gamma), \quad \text{for all } \gamma \in \mathbb{R}.$$
Also, \lmaref{finite_r-thmoment} implies that for each $r \in \mathbb{Z}_{>0}$, there exists $\delta = \delta(r) > 0$ such that
$$\sup_y \left\{ \int_{-\infty}^\infty |t|^{r+\delta} d\Phi_y(t) \right\} < \infty.$$
Combining the last two observations with Fact 4, we obtain
$$\lim_{x \rightarrow \infty} \textnormal{E} \bigg( \bigg( \frac{\mcs_y(f(\mfm)) - \textnormal{E}(\mcs_y)}{\sqrt{\textnormal{Var}(\mcs_y)}} \bigg)^r \bigg) = \mu_r,$$
and thus establish
$$\lim_{x \rightarrow \infty} F_y(\gamma) = \Phi(\gamma).$$
This completes the proof of \thmref{yurugen}, i.e., we obtain that for any $\gamma \in \mathbb{R}$, we have
$$\lim_{x \rightarrow \infty} \frac{1}{|\mcs(x)|} \bigg| \left\{ \mfm \in \mcs(x) \ : \ \frac{\omega(f(\mfm)) - \log \log N(\mfm)}{\sqrt{\log \log N(\mfm)}} \leq \gamma\right\} \bigg| = \Phi(\gamma).$$
\end{proof}
Next, as an application of \thmref{yurugen}, we prove a weaker version of the general result as the following:
\begin{proof}[\textbf{Proof of \thmref{less-general}}]
Since the subset $\mcs$ satisfy
$$|\mathcal{S}(x)| = C_\beta x^{\beta} + O_\beta \big( x^{\xi \beta} \big),$$
for some $0 \leq \xi < 1$, thus 
$$\frac{|\mathcal{S}(x^{1/2})|}{|\mathcal{S}(x)|} \ll \frac{1}{x^{\beta/2}},$$
and hence, $|\mathcal{S}(x^{1/2})| = o \left( |\mathcal{S}(x)| \right)$ is satisfied. 

Moreover, since
$$\mathcal{S}_\mfp(x) := \left\{ \mathfrak{m} \in \mcs(x) \ : \ n_\mfp (f(\mathfrak{m})) \geq 1 \right\},$$
satisfy \eqref{npx2A} given as
\begin{align*}
    |\mathcal{S}_\mfp(x)| 
    & = \frac{C_{\beta} x^\beta}{N(\mfp)} + \frac{C_{\mfp,\beta}' x^\beta}{N(\mfp)^{1+\eta}} + O_\beta \left( \frac{x^{\xi \beta}}{N(\mfp)^\xi}\right),
\end{align*}
for some $\eta > 0$ and where the constant $C_{\mfp,\beta}'$ is uniformly bounded in $\mfp$, thus,
\begin{align*}
    \frac{|\mathcal{S}_\mfp(x)|}{|\mathcal{S}(x)|} 
    & = \lambda_\mfp + e_\mfp(x),
\end{align*}
where $\lambda_\mfp = \frac{1}{N(\mfp)} + \frac{C_{\mfp,\beta}'}{ C_\beta} \frac{1}{N(\mfp)^{1+\eta}}$ and $e_\mfp(x) = O_h \left( \frac{1}{x^{(1 - \xi)\beta} N(\mfp)^{\xi}} \right)$. 

Next, we choose $y = x^\frac{\beta}{\log \log x} < x^\beta$, and check again that all the conditions in \thmref{yurugen} hold true. Note that the set in Condition (a) is empty and thus the condition holds trivially. By Part 4 of \lmaref{boundnm}, we obtain 
\begin{align*}
    \sum_{x^\frac{\beta}{\log \log x} < N(\mfp) \leq x^{\beta}} \lambda_\mfp \ll \sum_{x^\frac{\beta}{\log \log x} < N(\mfp) \leq x^{\beta}} \frac{1}{N(\mfp)} \ll \log \log \log x,
\end{align*}
which makes Condition (b) true. Using Part 1 of \lmaref{boundnm}, we have
$$\sum_{x^\frac{\beta}{\log \log x} < N(\mfp) \leq x^{\beta}} |e_\mfp(x)|  \ll_h \frac{1}{x^{(1-\xi)\beta}} \sum_{N(\mfp) \leq x^{\beta}} \frac{1}{N(\mfp)^{\xi}} \ll_k \frac{1}{\log x},$$
which makes Condition (c) true. Moreover, by Parts 3 and 4 of \lmaref{boundnm} again, we obtain
\begin{align*}
\sum_{N(\mfp) \leq x^\frac{\beta}{\log \log x}} \lambda_\mfp & = \sum_{N(\mfp) \leq x^\frac{\beta}{\log \log x}} \frac{1}{N(\mfp)} + O(1) \\
& = \log \log x + O (\log \log \log x),
\end{align*}
which makes Condition (d) true. Finally, again using Part 3 of \lmaref{boundnm} with $\alpha =2$, we have
$$\sum_{N(\mfp) \leq x^\frac{\beta}{\log \log x}} \lambda_\mfp^2 \ll \sum_{N(\mfp) \leq x^\frac{\beta}{\log \log x}} \frac{1}{N(\mfp)^2} \ll O(1).$$
This makes Condition (e) true. Finally, we are only required to verify Condition (f). Using \eqref{npx2A} and the Chinese Remainder Theorem, we obtain, for distinct prime ideals $\mfp_{\scaleto{1}{3pt}},\cdots,\mfp_{\scaleto{u}{3pt}}$, 
\begin{align*}
    & \left| \left\{ \mathfrak{m} \in \mathcal{S}(x) \ : \ n_{\mfp_i} (f(\mathfrak{m})) \geq 1 \text{ for all } i \in \{ 1, 2 ,\cdots, u\} \right\} \right| \\
    & = \left( \prod_{i=1}^u \left( C_\beta + \frac{C_{\mfp_i,\beta}'}{N(\mfp_i)^\eta} \right) \frac{1}{N(\mfp_i)} \right) x^{\beta} + O_h \left( \frac{x^{\xi \beta}}{\prod_{i=1}^u N(\mfp_i)^{\xi}} \right).
\end{align*}
Thus
\begin{align*}
    & \frac{\left| \left\{ \mathfrak{m} \in \mathcal{S}(x) \ : \ n_{\mfp_i} (f(\mathfrak{m})) \geq 1 \text{ for all } i \in \{ 1, 2 ,\cdots, u\} \right\} \right|}{|\mathcal{S}(x)|} \\
    & = \left( \prod_{i=1}^u \left( 1 + \frac{C_{\mfp_i,\beta}'}{C_\beta N(\mfp_i)^\eta} \right) \frac{1}{N(\mfp_i)} \right) + e_{\mfp_{\scaleto{1}{3pt}} \cdots \mfp_{\scaleto{u}{3pt}}}(x),
\end{align*}
where
$$|e_{\mfp_{\scaleto{1}{3pt}} \cdots \mfp_{\scaleto{u}{3pt}}}(x)| \ll_\beta \frac{1}{x^{(1 - \xi)\beta}}  \frac{1}{\prod_{i=1}^u N(\mfp_i)^{\xi}}.$$
Let $r \in \mathbb{Z}_{>0}$. By the definition of $\sum{\vphantom{\sum}}''$ in the conditions mentioned before \thmref{yurugen}, and using Part 1 of \lmaref{boundnm} and $x^{(1-\xi)\beta/\log \log x} = o(x^\epsilon)$ for any small $\epsilon > 0$, we have
\begin{align*}
    \sum{\vphantom{\sum}}'' |e_{\mfp_{\scaleto{1}{3pt}} \cdots \mfp_{\scaleto{u}{3pt}}}(x)| & \ll_k \frac{1}{x^{(1 - \xi)\beta}}  \left( \sum_{N(\mfp) \leq x^\frac{\beta}{\log \log x} } \frac{1}{N(\mfp)^{\xi}} \right)^u 
    \ll_k \frac{1}{x^{((1-\xi)\beta) - \epsilon'}},
\end{align*}
for any small $\epsilon' > 0$. Since, $x^{-(((1-\xi)\beta) - \epsilon')} = o \left( (\log \log x)^{-r/2} \right)$, thus Condition (f) holds true as well. Since all the conditions of \thmref{yurugen} hold with $y = x^{\beta/\log \log x}$, thus applying \thmref{yurugen} completes the proof.
\end{proof}

\section{The Erd\H{o}s-Kac theorems over \texorpdfstring{$h$}{}-free and \texorpdfstring{$h$}{}-full elements}
In this section, we prove the Erd\H{o}s-Kac theorem for $\omega(\mfm)$ over $h$-free and $h$-full elements.
We intend to prove the case of $h$-free elements given in \thmref{erdoskacforomega} as an application of \thmref{less-general}. We prove:
\begin{proof}[\textbf{Proof of \thmref{erdoskacforomega}}]
Consider the set $\mathcal{S} = \mathcal{S}_h$. By \eqref{hfreeidealcount} and \rmkref{remark1}, we have
$$|\mathcal{S}(x)| = \frac{\kappa}{\zeta_\mcm(h)} x + O \big( R_{\mathcal{S}_h}(x) \big),$$
where $R_{\mathcal{S}_h}(x)$ is defined in \eqref{RSh(x)} and satisfies $R_{\mathcal{S}_h}(x) \ll x^\tau$ where $0 \leq \tau< 1$. 

Let $f$ be the identity map. For a fixed prime element $\mfp$, let 
$$\mathcal{S}_\mfp(x) := \left\{ \mathfrak{m} \in \mathcal{S}_h (x) \ : \ n_\mfp(f(\mathfrak{m})) \geq 1 \right\} = \left\{ \mathfrak{m} \in \mathcal{S}_h (x) \ : \ n_\mfp(\mathfrak{m}) \geq 1 \right\}.$$
Using \lmaref{hfreeidealrestrict} with $R_{\mathcal{S}_h}(x) \ll x^\tau$, and the identity
$$ \frac{N(\mfp)^{h-1} - 1}{N(\mfp)^h - 1} = \frac{1}{N(\mfp)} - \frac{N(\mfp)-1}{N(\mfp)(N(\mfp)^h -1)},$$
we obtain
\begin{align*}
    |\mathcal{S}_\mfp(x)| 
    & = \sum_{k=1}^{h-1} |\mathcal{S}_{h,\mfp} (x/N(\mfp)^k)| \notag \\
    & = \sum_{k=1}^{h-1} \left( \left( \frac{N(\mfp)^h - N(\mfp)^{h-1})}{N(\mfp)^h - 1} \right) \frac{1}{N(\mfp)^k} \frac{\kappa}{\zeta_\mcm(h)} x + O \left( R_{\mathcal{S}_h}(x/N(\mfp)^k) \right) \right) \notag \\
    & = \frac{N(\mfp)^{h-1} - 1}{N(\mfp)^h - 1} \frac{\kappa}{\zeta_\mcm(h)} x + O \left( \left( \frac{x}{N(\mfp)} \right)^\tau \right) \notag \\
    & = \left( \frac{\kappa}{\zeta_\mcm(h)} -  \frac{N(\mfp)-1}{N(\mfp)^h -1} \frac{\kappa}{\zeta_\mcm(h)} \right) \frac{x}{N(\mfp)} + O \left( \left( \frac{x^\tau}{N(\mfp)^\tau} \right) \right).
\end{align*}

Thus, $\mcs(x)$ and $\mcs_\mfp(x)$ satisfy the conditions of \thmref{less-general} with $f = \text{identity}$, $\mathcal{S} = \mathcal{S}_h$, $\beta= \eta = 1$, $\xi = \tau$, $C_\beta = \kappa/\zeta_\mcm(h)$, and
$$C_{\mfp,\beta}' = - C_\beta \frac{N(\mfp)^2-N(\mfp)}{N(\mfp)^h -1} \quad \text{ with } \quad |C_{\mfp,\beta}'| \leq C_\beta.$$
Thus applying \thmref{less-general} completes the proof.
\end{proof}
Next, for the case of $h$-full elements given in \thmref{erdoskacforomegahfull}, we prove:

\begin{proof}[\textbf{Proof of \thmref{erdoskacforomegahfull}}]
Consider the set $\mathcal{S} = \mathcal{N}_h$. By \eqref{hfullideals} and \rmkref{remark2}, we have
$$|\mathcal{S}(x)| = \kappa \gamma_{\scaleto{h}{4.5pt}} x^{1/h} + O_h \big( R_{\mathcal{N}_h}(x) \big),$$
where $R_{\mathcal{N}_h}(x) \ll x^{\nu/h}$ for some $0 \leq \nu < 1$.
Let $f$ be the identity map. For a fixed prime ideal $\mfp$, let 
$$\mathcal{S}_\mfp(x) := \left\{ \mathfrak{m} \in \mathcal{N}_h(x) \ : \ n_\mfp (f(\mathfrak{m})) \geq 1 \right\} = \left\{ \mathfrak{m} \in \mathcal{N}_h(x) \ : \ n_\mfp (\mathfrak{m}) \geq 1 \right\}.$$
Using \lmaref{hfullidealsrestrict} with $R_{\mathcal{N}_h}(x) \ll x^{\nu/h}$ and the identity
$$ \frac{1}{N(\mfp)(1- N(\mfp)^{-1/h} + N(\mfp)^{-1})} = \frac{1}{N(\mfp)} + \frac{N(\mfp)^{-1/h}-N(\mfp)^{-1}}{N(\mfp)(1 - N(\mfp)^{-1/h} + N(\mfp)^{-1})},$$
we obtain
\begin{align*}
    |\mathcal{S}_\mfp(x)| 
    & = \sum_{k=h}^{\left\lfloor \frac{\log x}{\log N(\mfp)} \right\rfloor} \left| \mathcal{N}_{h,\mfp}(x/N(\mfp)^k) \right|  \notag\\
    & = \left( \sum_{k=h}^{\infty} \frac{1}{N(\mfp)^{k/h}} \right)\frac{\kappa \gamma_{\scaleto{h}{4.5pt}}}{\left( 1+ \frac{N(\mfp)^{-1}}{1-N(\mfp)^{-1/h}} \right)} x^{1/h} + O_h \left( x^{\nu/h}   \sum_{k=h}^{\left\lfloor \frac{\log x}{\log N(\mfp)} \right\rfloor}  \frac{1}{N(\mfp)^{k\nu/h}} \right) \notag \\
    & = \frac{\kappa \gamma_{\scaleto{h}{4.5pt}}}{N(\mfp) (1-N(\mfp)^{-1/h}+N(\mfp)^{-1})} x^{1/h} + O_h \left( \frac{x^{\nu/h}}{N(\mfp)^{\nu}} \right) \\
    & = \left( \kappa \gamma_{\scaleto{h}{4.5pt}} + \frac{ \kappa \gamma_{\scaleto{h}{4.5pt}} (1-N(\mfp)^{-1+(1/h)})}{N(\mfp)^{1/h}(1 - N(\mfp)^{-1/h} + N(\mfp)^{-1})}\right)\frac{x^{1/h}}{N(\mfp)} + O_h \left( \frac{x^{\nu/h}}{N(\mfp)^{\nu}} \right) 
\end{align*}

Thus, $\mcs(x)$ and $\mcs_\mfp(x)$ satisfy the conditions of \thmref{less-general} with $f = \text{identity}$, $\mathcal{S} = \mathcal{N}_h$, $\beta= \eta = 1/h$, $\xi = \nu$, $C_\beta = \kappa \gamma_{\scaleto{h}{4.5pt}}$, and
$$C_{\mfp,\beta}' = C_\beta \frac{1-N(\mfp)^{-(h-1)/h}}{1 - N(\mfp)^{-1/h} + N(\mfp)^{-1}} \quad \text{ with } \quad |C_{\mfp,\beta}'| \leq C_\beta \frac{1-2^{-(h-1)/h}}{1 - 2^{-1/h} + 2^{-1}}.$$
Thus applying \thmref{less-general} completes the proof.
\end{proof}

\section{Other generalizations of the Erd\H{o}s-Kac theorem}

For an integer $k \geq 1$, recall that $\omega_k(\mfm)$ counts the distinct prime elements generating $\mfm$ with multiplicity $k$.

For an element $\mfm \in \mcm$ and an integer $k \geq 1$, let $\mfm_k$ be defined in \eqref{mk} as
\begin{equation*}
    \mfm_k = k \cdot \sum_{\substack{\mfp \\ n_\mfp(\mfm) = k}} \mfp.
\end{equation*}
We define the map $f_k : \mcs \rightarrow \mcm$ as
$$f_k(\mfm) = \mfm_k.$$
Recall that, we have
$$\omega(f_k(\mfm)) = \omega_k(\mfm).$$

For a sequence $\mathcal{A} = (a_1, a_2, \cdots)$ of complex numbers, recall that $\omega_\mcA: \mcm \rightarrow \mathbb{R}$ is defined as
\begin{equation}\label{def-omegaA}
    \omega_\mcA (\mfm) = \sum_{k \geq 1} a_k \ \omega(f_k(\mfm)) = \sum_{k \geq 1} a_k \ \omega_k(\mfm),
\end{equation}
where the sum is finite for each $\mfm$. In this section, we prove the following generalizations of the Erd\H{o}s-Kac theorem in the order mentioned:

\begin{enumerate}
    \item Erd\H{o}s-Kac theorem for $\omega_1(\mfm)$ over $h$-free elements with $h \geq 2$,
    \item Erd\H{o}s-Kac theorem for $\omega_k(\mfm)$ over $k$-full elements with $k \geq 1$, and
    \item if $a_k \neq 0$, then $\frac{1}{a_k}\omega_\mcA$ satisfies the Erd\H{o}s-Kac theorem over $k$-full elements.
\end{enumerate}

The first two results are proved as applications of \thmref{less-general}, and the final result is deduced from the first two results.

We intend to prove the Erd\H{o}s-Kac theorem for $\omega_1(\mfm)$ over $h$-free elements given in \thmref{erdoskacforomega1} by applying \thmref{yurugen} with $f = f_1$, $\mcs = \mathcal{\mcs}_h$, $\beta = \eta = 1$, and $y = x^{1/\log \log x}$. We prove:
\begin{proof}[\textbf{Proof of \thmref{erdoskacforomega1}}]
Consider the set $\mathcal{S} = \mathcal{S}_h$. Recall that, by \eqref{hfreeidealcount} and \lmaref{remark1}, we have
$$|\mathcal{S}(x)| = \frac{\kappa}{\zeta_\mcm(h)} x + O_h \big( R_{\mathcal{S}_h}(x) \big),$$
where $R_{\mathcal{S}_h}(x) \ll x^\tau$ where $0 \leq \tau  < 1$. For a fixed prime element $\mfp$, let 
$$\mathcal{S}_\mfp(x) := \left\{ \mathfrak{m} \in \mathcal{S}_h (x) \ : \ n_\mfp(f_1(\mathfrak{m})) \geq 1 \right\} = \left\{ \mathfrak{m} \in \mathcal{S}_h (x) \ : \ n_\mfp(\mathfrak{m}) = 1 \right\}.$$
Using \lmaref{hfreeidealrestrict} with $R_{\mathcal{S}_h}(x) \ll x^\tau$, and the identity
$$ \frac{N(\mfp)^{h} - N(\mfp)^{h-1}}{N(\mfp)^h - 1} = 1 -  \frac{N(\mfp)^{h-1}-1}{N(\mfp)^h -1},$$
we obtain
\begin{align*}
    |\mathcal{S}_\mfp(x)| 
    & = |\mathcal{S}_{h,\mfp} (x/N(\mfp))| \notag \\
    & = \left( \frac{N(\mfp)^h - N(\mfp)^{h-1}}{N(\mfp)^h - 1} \right) \frac{1}{N(\mfp)} \frac{\kappa}{\zeta_\mcm(h)} x + O_h \left( R_{\mathcal{S}_h}(x/N(\mfp)) \right) \notag \\
    & = \left( \frac{\kappa}{\zeta_\mcm(h)} - \frac{\kappa}{\zeta_\mcm(h)} \frac{N(\mfp)^{h-1}-1}{N(\mfp)^h -1} \right) \frac{x}{N(\mfp)} + O_h \left( \left( \frac{x}{N(\mfp)} \right)^\tau \right).
\end{align*}

Thus, $\mcs(x)$ and $\mcs_\mfp(x)$ satisfy the conditions of \thmref{less-general} with $f = f_1$, $\mathcal{S} = \mathcal{S}_h$, $\beta = \eta = 1$, $\xi = \tau$, $C_\beta = \kappa/\zeta_\mcm(h)$, and
$$C_{\mfp,\beta}' = - C_\beta \frac{N(\mfp)^h-N(\mfp)}{N(\mfp)^h -1} \quad \text{ with } \quad |C_{\mfp,\beta}'| \leq C_\beta.$$
Thus applying \thmref{less-general} and noticing that $\omega(f_1(\mfm)) = \omega_1(\mfm)$ completes the proof.
\end{proof}

Next, for an integer $k \geq 1$, we prove the Erd\H{o}s-Kac theorem for $\omega_k(\mfm)$ over $k$-full elements given in \thmref{erdoskacforomegah} by applying \thmref{yurugen} with $f = f_k$, $\mcs = \mathcal{N}_k$, $\beta = \eta = 1/k$, and $y = x^{1/k \log \log x}$. We prove:

\begin{proof}[\textbf{Proof of \thmref{erdoskacforomegah}}]

First, we deal with the case $k = 1$, i.e, when $\mcs = \mathcal{N}_1 = \mcm$. By condition \eqref{star}, we have
$$|\mathcal{S}(x)| = \kappa x + O \big( x^\theta \big),$$
where $0 \leq \theta < 1$. For a fixed prime element $\mfp$, let 
$$\mathcal{S}_\mfp(x) := \left\{ \mathfrak{m} \in \mcm(x) \ : \ n_\mfp(f_1(\mathfrak{m})) \geq 1 \right\} = \left\{ \mathfrak{m} \in \mcm (x) \ : \ n_\mfp(\mathfrak{m}) = 1 \right\}.$$
Using \eqref{star}, we have
\begin{align*}
    |\mathcal{S}_\mfp(x)| = |\mathcal{S}(x/N(\mfp))| - |\mathcal{S}(x/N(\mfp)^2)|
   = \left( \kappa - \frac{\kappa}{N(\mfp)} \right) \frac{x}{N(\mfp)} + O\left( \left( \frac{x}{N(\mfp)} \right)^\theta \right).
\end{align*}

Thus, $\mcs(x)$ and $\mcs_\mfp(x)$ satisfy the conditions of \thmref{less-general} with $f = f_1$, $\mathcal{S} = \mathcal{M} =  \mathcal{N}_1$, $\beta = \eta = 1$, $\xi = \theta$, and $C_\beta = C_{\mfp,\beta}' = \kappa$.
Thus applying \thmref{less-general}, for $\gamma \in \mathbb{R}$, we have
$$\lim_{x \rightarrow \infty} \frac{1}{|\mathcal{N}_1(x)|} \left| \left\{ \mfm \in \mathcal{N}_1(x) \ : \ N(\mfm) \geq 3, \ \frac{\omega(f_1(\mfm)) - \log \log N(\mfm)}{\sqrt{\log \log N(\mfm)}} \leq \gamma\right\} \right| = \Phi(\gamma).$$
Finally, noticing that $\omega(f_1(\mfm)) = \omega_1(\mfm)$ completes the proof for the case $k = 1$.

Next, we consider the case $k = h \geq 2$, i.e., the case of $h$-full elements. Consider the set $\mathcal{S} = \mathcal{N}_h$. By \eqref{hfullideals} and \rmkref{remark2}, we have
$$|\mathcal{S}(x)| = \kappa \gamma_{\scaleto{h}{4.5pt}} x^{1/h} + O_h \big( R_{\mathcal{N}_h}(x) \big),$$
where $R_{\mathcal{N}_h}(x) \ll x^{\nu/h}$ for some $0 < \nu < 1$.
For a fixed prime ideal $\mfp$, let 
$$\mathcal{S}_\mfp(x) := \left\{ \mathfrak{m} \in \mathcal{N}_h(x) \ : \ n_\mfp (f_h(\mathfrak{m})) \geq 1 \right\} = \left\{ \mathfrak{m} \in \mathcal{N}_h(x) \ : \ n_\mfp (\mathfrak{m}) = h \right\}.$$
Using \lmaref{hfullidealsrestrict} with $R_{\mathcal{N}_h}(x) \ll x^{\nu/h}$, and the identity
$$ \frac{1-N(\mfp)^{-1/h}}{N(\mfp) (1-N(\mfp)^{-1/h}+N(\mfp)^{-1})} = \frac{1}{N(\mfp)} - \frac{N(\mfp)^{-1}}{N(\mfp)(1 - N(\mfp)^{-1/h} + N(\mfp)^{-1})},$$
we obtain
\begin{align*}
    |\mathcal{S}_\mfp(x)| 
    & = \left| \mathcal{N}_{h,\mfp}(x/N(\mfp)^h) \right|  \notag\\
    & = \frac{\kappa \gamma_{\scaleto{h}{4.5pt}}}{N(\mfp) \left( 1+ \frac{N(\mfp)^{-1}}{1-N(\mfp)^{-1/h}} \right)} x^{1/h} + O_h \left( \frac{x^{\nu/h}}{N(\mfp)^{\nu}} \right) \notag \\
    & = \frac{\left( 1-N(\mfp)^{-1/h} \right) \kappa \gamma_{\scaleto{h}{4.5pt}}}{N(\mfp) (1-N(\mfp)^{-1/h}+N(\mfp)^{-1})} x^{1/h} + O_h \left( \frac{x^{\nu/h}}{N(\mfp)^{\nu}} \right) \\
    & = \left( \kappa \gamma_{\scaleto{h}{4.5pt}} - \frac{\kappa \gamma_{\scaleto{h}{4.5pt}}}{N(\mfp)^{1/h}} \frac{N(\mfp)^{-1+(1/h)}}{1 - N(\mfp)^{-1/h} + N(\mfp)^{-1}}\right) \frac{x^{1/h}}{N(\mfp)} + O_h \left( \frac{x^{\nu/h}}{N(\mfp)^{\nu}} \right).
\end{align*}
Thus, $\mcs(x)$ and $\mcs_\mfp(x)$ satisfy the conditions of \thmref{less-general} with $f = f_h$, $\mathcal{S} = \mathcal{N}_h$, $\beta = \eta = 1/h$, $\xi = \nu$, $C_\beta = \kappa \gamma_{\scaleto{h}{4.5pt}}$, and
$$C_{\mfp, \beta}' = - C_\beta \frac{N(\mfp)^{-(h-1)/h}}{1 - N(\mfp)^{-1/h} + N(\mfp)^{-1}} \quad \text{ with } \quad |C_{\mfp,\beta}'| \leq C_\beta \frac{2^{-(h-1)/h}}{1 - 2^{-1/h} + 2^{-1}}.$$
Thus applying \thmref{less-general} and using $\omega(f_h(\mfm)) = \omega_h(\mfm)$ establishes the announced result for $k = h \geq 2$. This completes the proof.
\end{proof}

Next, for a sequence $\mcA = (a_1, a_2, \ldots)$ with $a_k \neq 0$, we prove that $\frac{1}{a_k}\omega_\mcA$ satisfies the Erd\H{o}s-Kac theorem over $k$-full elements in $\mcm$: 
\begin{proof}[\textbf{Proof of \thmref{erdoskacforomegaA}}]
    For a function $g : \mcm \rightarrow \mathbb{C}$ and an element $\mfm \in \mcm$ with $N(\mfm) \geq 3$, let $G_g(\mfm)$ be the ratio
\begin{equation}\label{rfn}
    G_g(\mfm) := \frac{g(\mfm) - \log \log N(\mfm)}{\sqrt{\log \log N(\mfm)}}.
\end{equation}
By hypothesis, $a_k \neq 0$. In this proof, we will use $g$ to represent $\frac{1}{a_k} \omega_\mcA$ or $\omega_k$ when necessary. For $a \in \mathbb{R}$ and a subset $\mcs$ of $\mcm$, let $\mcs(x)$ denote the set of elements of $\mcs$ with norm less than or equal to $x$, and 
\begin{equation}\label{dfsxa}
    D(g,\mcs,x,a) := \frac{1}{|\mcs(x)|} \left| \{ \mfm \in \mcs(x) \ : \ G_g(\mfm) \leq a \} \right|
\end{equation}
be the density function for sufficiently large $x$. Note that, if $k \geq 1$ and $\mfm$ is a $k$-full element, then $\omega_i(\mfm) = 0$ for all $i \in {1, 2, \ldots, k-1}$. Thus, $\frac{1}{a_k}\omega_\mcA(\mfm) = \omega_k(\mfm) + \sum_{i \geq k+1} \frac{a_i}{a_k} \omega_i(\mfm)$. Moreover, by \thmref{erdoskacforomegah}, we have
$$\lim_{x \rightarrow \infty} D(\omega_k, \mathcal{N}_k,x,a) = \Phi(a).$$
We intend to show
$$\lim_{x \rightarrow \infty} D\left( \frac{1}{a_k} \omega_\mcA, \mathcal{N}_k,x,a \right) = \Phi(a).$$

For any $\epsilon > 0$, we define the set
$$A(\mathcal{N}_k,x,\epsilon) := \left\{ \mfm \in \mathcal{N}_k(x) \ : \ \frac{\left| \frac{1}{a_k} \omega_\mcA(\mfm) - \omega_k(\mfm) \right|}{\sqrt{\log \log N(\mfm)}} \leq \epsilon \right\}.$$

Let $A^c(\mathcal{N}_k,x,\epsilon)$ denote the complement of $A(\mathcal{N}_k,x,\epsilon)$ inside $\mathcal{N}_k(x)$. We first deduce that $|A^c(\mathcal{N}_k,x,\epsilon)| = o(\mathcal{N}_k(x))$. Notice that,
\begin{align*}
    & \sum_{\substack{ \mfm \in A^c(\mathcal{N}_k,x,\epsilon) \\ x/\log x \leq N(\mfm)}} \left| \frac{1}{a_k} \omega_\mcA(\mfm) - \omega_k(\mfm) \right| \\
    & \geq \epsilon \sqrt{\log \log (x/\log x)} \ |\{ \mfm \in A^c(\mathcal{N}_k,x,\epsilon) \ | \ N(\mfm) \geq x/\log x \}|.
\end{align*}
Moreover
\begin{align*}
    \sum_{\substack{ \mfm \in A^c(\mathcal{N}_k,x,\epsilon) \\ x/\log x \leq N(\mfm)}} \left| \frac{1}{a_k} \omega_\mcA(\mfm) - \omega_k(\mfm) \right| & \leq \sum_{\mfm \in \mathcal{N}_k(x)} \left| \frac{1}{a_k} \omega_\mcA(\mfm) - \omega_k(\mfm) \right| \\
    & \leq \sum_{\mfm \in \mathcal{N}_k(x)} \left( \sum_{i \geq k+1} \left| \frac{a_i}{a_k} \right| \omega_i(\mfm) \right).
\end{align*}
Moreover, by condition \eqref{star} for $k =1$ and \eqref{hfullideals} for $k \geq 2$, we have
$$\sum_{\substack{\mfm \in \mathcal{N}_k(x) \\ n_\mfp(\mfm) \geq i}} 1 \ll \sum_{\substack{\mfm \in \mathcal{N}_k(x/N(\mfp)^i)}} 1 \ll \frac{x^{1/k}}{N(\mfp)^{i/k}}.$$
Note that, the rate of growth of $a_i'$s given in the hypothesis of the theorem ensures that \eqref{Condition-ai} holds. Thus, by interchanging sums and applying \eqref{Condition-ai}, we obtain
\begin{align*}
    \sum_{\mfm \in \mathcal{N}_k(x)} \left( \sum_{i \geq k+1} \left| \frac{a_i}{a_k} \right| \omega_i(\mfm) \right) & = \sum_{\mfm \in \mathcal{N}_k(x)} \left( \sum_{i \geq k+1} \left| \frac{a_i}{a_k} \right| \sum_{\substack{\mfp \\ n_\mfp(\mfm) = i}} 1 \right) \\
    & = \sum_{i \geq k+1} \left| \frac{a_i}{a_k} \right| \sum_{\mfm \in \mathcal{N}_k(x)} \sum_{\substack{\mfp \\ n_\mfp(\mfm) = i}} 1 \\
    & \ll \sum_{i \geq k+1} \left| \frac{a_i}{a_k} \right| \sum_{\substack{\mfp \\ N(\mfp) \leq x^{1/k}}} \sum_{\substack{\mfm \in \mathcal{N}_k(x) \\ n_\mfp(\mfm) \geq i}} 1 \\
    & \ll x^{1/k} \sum_{\substack{\mfp \\ N(\mfp) \leq x^{1/k}}} \sum_{i \geq k+1} \frac{|a_i|}{N(\mfp)^{i/k}} \\
    & \ll x^{1/k}.
\end{align*}
Combining the above results, we obtain
\begin{align*}
    & \epsilon \sqrt{\log \log (x/\log x)} \ |\{ \mfm \in A^c(\mathcal{N}_k,x,\epsilon) \ | \ N(\mfm) \geq x/\log x \}| \\
    & \ll
    \sum_{\substack{\mfm \in A^c(\mathcal{N}_k,x,\epsilon) \\ x/\log x \leq N(\mfm)}} \left| \frac{1}{a_k} \omega_\mcA(\mfm) - \omega_k(\mfm) \right|  \\
    & \ll x^{1/k}.
\end{align*}
Thus,
$$|\{ \mfm \in A^c(\mathcal{N}_k,x,\epsilon) \ | \ N(\mfm) \geq x/\log x \}| \ll \frac{x^{1/k}}{\epsilon \sqrt{\log \log (x/\log x)}} = o(x^{1/k}).$$
Moreover, by \eqref{hfullideals},
$$|\{ \mfm \in A^c(\mathcal{N}_k,x,\epsilon) \ | \ N(\mfm) < x/\log x \}| \ll \mathcal{N}_k(x/\log x) \ll_k \frac{x^{1/k}}{\log x} = o(x^{1/k}).$$
Combining the above two results, we deduce
$$|A^c(\mathcal{N}_k,x,\epsilon)| = o(x^{1/k}).$$
Since, by Condition $(\star)$ and \eqref{hfullideals} again, $\mathcal{N}_k(x) \gg x^{1/k}$, thus $|A^c(\mathcal{N}_k,x,\epsilon)| = o(\mathcal{N}_k(x))$ follows.

Note that, if $\mfm \in A(\mathcal{N}_k,x,\epsilon)$, we have
$$G_{\omega_k}(\mfm) - \epsilon \leq G_{\frac{1}{a_k}\omega_\mcA}(\mfm) \leq G_{\omega_k}(\mfm) + \epsilon.$$
Thus, for any $a \in \mathbb{R}$, if $\mfm \in A(\mathcal{N}_k,x,\epsilon)$, we have
\begin{equation}\label{use1}
    G_{\omega_k}(\mfm) \leq a - \epsilon \implies G_{\frac{1}{a_k}\omega_\mcA}(\mfm) \leq a,
\end{equation}
and
\begin{equation}\label{use2}
    G_{\frac{1}{a_k}\omega_\mcA}(\mfm) \leq a \implies G_{\omega_k}(\mfm) \leq a + \epsilon.
\end{equation}
By \eqref{use1}, we have
$$ \{ \mfm \in A(\mathcal{N}_k,x,\epsilon) \ : \ G_{\omega_k}(\mfm) \leq a - \epsilon \} \subseteq  \{ \mfm \in A(\mathcal{N}_k,x,\epsilon) \ : \ G_{\frac{1}{a_k}\omega_\mcA}(\mfm) \leq a \},$$
which implies
\begin{multline*}
    \{ \mfm \in A(\mathcal{N}_k,x,\epsilon) \ : \ G_{\omega_k}(\mfm) \leq a - \epsilon \} \cup \{ \mfm \in A^c(\mathcal{N}_k,x,\epsilon) \ : \ G_{\omega_k}(\mfm) \leq a - \epsilon \} \\
    = \{ \mfm \in \mathcal{N}_k(x) \ : \ G_{\omega_k}(\mfm) \leq a - \epsilon \} \\
    \subseteq \{ \mfm \in \mathcal{N}_k(x) \ : \ G_{\frac{1}{a_k}\omega_\mcA}(\mfm) \leq a \} \cup \{ \mfm \in A^c(\mathcal{N}_k,x,\epsilon) \ : \ G_{\omega_k}(\mfm) \leq a - \epsilon \}.
\end{multline*}
Thus
$$|\{ \mfm \in \mathcal{N}_k(x) \ : \ G_{\omega_k}(\mfm) \leq a - \epsilon \}| \leq |\{ \mfm \in \mathcal{N}_k(x) \ : \ G_{\frac{1}{a_k}\omega_\mcA}(\mfm) \leq a \}| + |A^c(\mathcal{N}_k,x,\epsilon)|.$$
Therefore, by the definition of $D(g,\mathcal{N}_k,x,a)$, \thmref{erdoskacforomegah}, and the result $|A^c(\mathcal{N}_k,x,\epsilon)| = o(\mathcal{N}_k(x))$ above, we have
\begin{equation}\label{use3}
    \Phi(a-\epsilon) \leq \liminf_{x \rightarrow \infty} D\left(\frac{1}{a_k}\omega_\mcA,\mathcal{N}_k,x,a \right).
\end{equation}

Moreover, by \eqref{use2}, we have
$$ \{ \mfm \in A(\mathcal{N}_k,x,\epsilon) \ : \ G_{\frac{1}{a_k}\omega_\mcA}(\mfm) \leq a \} \subseteq  \{ \mfm \in A(\mathcal{N}_k,x,\epsilon) \ : \  G_{\omega_k}(\mfm) \leq a + \epsilon  \},$$
which implies
\begin{multline*}
     \{ \mfm \in A(\mathcal{N}_k,x,\epsilon) \ : \ G_{\frac{1}{a_k}\omega_\mcA}(\mfm) \leq a \} \cup \{ \mfm \in A^c(\mathcal{N}_k,x,\epsilon) \ : \ G_{\frac{1}{a_k}\omega_\mcA}(\mfm) \leq a \} \\
    = \{ \mfm \in \mathcal{N}_k(x) \ : \ G_{\frac{1}{a_k}\omega_\mcA}(\mfm) \leq a \} \\
    \subseteq \{ \mfm \in \mathcal{N}_k(x) \ : \  G_{\omega_k}(\mfm) \leq a + \epsilon  \} \cup \{ \mfm \in A^c(\mathcal{N}_k,x,\epsilon) \ : \ G_{\frac{1}{a_k}\omega_\mcA}(\mfm) \leq a \}.
\end{multline*}
Thus
$$|\{ \mfm \in \mathcal{N}_k(x) \ : \ G_{\frac{1}{a_k}\omega_\mcA}(\mfm) \leq a \}| \leq |\{ \mfm \in \mathcal{N}_k(x) \ : \  G_{\omega_k}(\mfm) \leq a + \epsilon  \}| + |A^c(\mathcal{N}_k,x,\epsilon)|.$$
Again, by the definition of $D(g,\mathcal{N}_k,x,a)$, \thmref{erdoskacforomegah}, and the result $|A^c(\mathcal{N}_k,x,\epsilon)| = o(\mathcal{N}_k(x))$ above, we have
\begin{equation}\label{use4}
    \limsup_{x \rightarrow \infty} D\left(\frac{1}{a_k}\omega_\mcA,\mathcal{N}_k,x,a \right) \leq \Phi(a+\epsilon).
\end{equation}
Combining \eqref{use3} and \eqref{use4}, we obtain
$$\Phi(a-\epsilon) \leq \liminf_{x \rightarrow \infty} D\left(\frac{1}{a_k}\omega_\mcA,\mathcal{N}_k,x,a \right) \leq \limsup_{x \rightarrow \infty} D\left(\frac{1}{a_k}\omega_\mcA,\mathcal{N}_k,x,a \right) \leq \Phi(a+\epsilon).$$
Since $\epsilon > 0$ is arbitrary, thus we obtain
$$\lim_{x \rightarrow \infty} D\left(\frac{1}{a_k}\omega_\mcA,\mathcal{N}_k,x,a \right) = \Phi(a).$$
This completes the proof.
\end{proof}

\section{Applications of the general setting}\label{applications}

In this section, we provide various applications of our general setting. In each case, we show that condition \eqref{star} holds, and thus deduce the Erd\H{o}s-Kac theorem for the $\omega_\mcA$-function over $h$-free and $k$-full elements, for some integer $h \geq 2$ and $k \geq 1$, and where $\mcA = (a_1, a_2, \ldots)$ satisfies some of the following types:

\begin{enumerate}
    \item[(1)] if $a_i = 1$ for all $i \in \mathbb{Z}_{>0}$, i.e., $\omega_\mcA(\mfm) = \omega(\mfm)$,
    \item[(2)] if $a_i = i$ for all $i \in \mathbb{Z}_{>0}$, i.e., $\omega_\mcA(\mfm) = \Omega(\mfm)$,
    \item[(3)] if $a_i = \log(i+1)$ for all $i \in \mathbb{Z}_{>0}$, i.e., $\omega_\mcA(\mfm) = \log d(\mfm)$,
    \item[(4)] if $a_i = 1$ for all odd $i$, and $a_i = - 1$ for all even $i$, i.e., $\omega_\mcA(\mfm) = \omega_T(\mfm)$ (see \eqref{def-omegaT}).
    \item[(5)] if $a_i = 0$ for all $i \neq k$ and $a_k = 1$, i.e., $\omega_\mcA(\mfm) = \omega_k(\mfm)$.
\end{enumerate}

\subsection{The case of ideals in number fields}

Let $K/\mathbb{Q}$ be a number field of degree $n_K = [K: \mathbb{Q}]$ and $\mathcal{O}_K$ be its ring of integers. Let $\mathcal{P}$ be the set of prime ideals of $\mathcal{O}_K$ and $\mathcal{M}$ be the set of ideals of $\mathcal{O}_K$. Let the norm map be $N: \mathcal{M} \rightarrow \mathbb{Z}_{>0}$ be the standard norm map, i.e., $\mathfrak{m} \mapsto N(\mathfrak{m}) := |\mathcal{O}_K/ \mathfrak{m}|$. Let $X = \mathbb{Q}$.

 Let $\kappa_K$ be given by
\begin{equation*}
    \kappa_K = \frac{2^{r_1}(2 \pi)^{r_2} h R}{\nu \sqrt{|d_K|}},
\end{equation*}
with
\begin{align*}
    r_1 & = \text{the number of real embeddings of } K,\\
    2 r_2 & = \text{the number of complex embeddings of } K,\\
    h & = \text{the class number},\\
    R & = \text{the regulator}, \\
    \nu & = \text{the number of roots of unity}, \\
    d_K & = \text{the discriminant of } K.
\end{align*}
Landau in \cite[Satz 210]{Landau} proved that
\begin{equation*}
    \sum_{\substack{\mfm \in \mcm \\ N(\mfm) \leq x}} 1 = \kappa_K x + O \left( x^{1 - \frac{2}{n_{\scaleto{K}{3pt}}+1}}\right),
\end{equation*}
which satisfies condition \eqref{star} with $\kappa = \kappa_K$ and $\theta = 1 - \frac{2}{n_{\scaleto{K}{3pt}}+1}$. Thus, \thmref{erdoskacforomegaA-hfree} and \thmref{erdoskacforomegaA} give the Erd\H{o}s-Kac theorem for $\omega_\mcA(\mfm)$ over $h$-free and $k$-full ideals respectively as the following:

\begin{cor}
Let $x > 2$ be a rational number. Let $h \geq 2$ be an integer. Let $\mathcal{S}_h(x)$ be the set of $h$-free ideals with norm less than or equal to $x$. Let $\mcA$ be any sequence from Types 1-4 and Type 5 with $k = 1$. Then for $a \in \mathbb{R}$, we have
$$\lim_{x \rightarrow \infty} \frac{1}{|\mathcal{S}_h(x)|} \bigg| \left\{ \mfm \in \mathcal{S}_h(x) \ : \ |\mathcal{O}_K/ \mathfrak{m}| \geq 3, 
\ \frac{\frac{1}{a_1} \omega_\mcA(\mfm) - \log \log |\mathcal{O}_K/ \mathfrak{m}|}{\sqrt{\log \log |\mathcal{O}_K/ \mathfrak{m}|}} \leq a \right\} \bigg| = \Phi(a).$$
\end{cor} 
\begin{cor}
Let $x > 2$ be a rational number. Let $k \geq 1$ be an integer. Let $\mathcal{N}_k(x)$ be the set of $k$-full ideals with norm less than or equal to $x$. Let $\mcA$ be any sequence from Types 1-5. Then for $a \in \mathbb{R}$, we have
$$\lim_{x \rightarrow \infty} \frac{1}{|\mathcal{N}_k(x)|} \bigg| \left\{ \mfm \in \mathcal{N}_k(x) \ : \ |\mathcal{O}_K/ \mathfrak{m}| \geq 3, 
\ \frac{\frac{1}{a_k} \omega_\mcA(\mfm) - \log \log |\mathcal{O}_K/ \mathfrak{m}|}{\sqrt{\log \log |\mathcal{O}_K/ \mathfrak{m}|}} \leq a \right\} \bigg| = \Phi(a).$$
\end{cor}
\begin{rmk}
    The Erdos-Kac theorems for $\omega(\mfm)$ over $h$-free and $h$-full ideals were first proved in \cite[Theorems 1.3 \& 1.4]{dkl3}, which employed a similar proof strategy as in this article.
\end{rmk}

\subsection{The case of effective divisors in global function fields}

Let $q$ be a prime power and $\mathbb{F}_q$ be the finite field with $q$ elements. Let $K/\mathbb{F}_q$ be a global function field. Let $G_K$ be its genus and $C_K$ be its class number. A prime $\mfp$ in $K$ is a discrete valuation ring $R$ with maximal ideal $P$ such that $P \subset R$ and the quotient field of $R$ is $K$. The degree of $\mfp$, denoted as $\deg \mfp$, is defined as the dimension of $R/P$ over $\mathbb{F}_q$, which is finite. Let $\mcp$ be the set of all primes in $K$. Let $\mcm$ be the free abelian monoid generated by $\mcp$. More precisely, for each $\mfm \in \mathcal{M}$, we write
$$\mfm = \sum_{\mfp \in  \mathcal{P}} n_\mfp(\mfm) \mfp,$$
with $n_\mfp(\mfm) \in \mathbb{Z}_{>0} \cup \{ 0 \}$ and $n_\mfp(\mfm) =0$ for all but finitely many $\mfp$. We call elements in $\mcm$ as effective divisors. For an element $\mfm \in \mcm$, we define the degree of $\mfm$ as
$$\deg \mfm = \sum_{\mfp \in  \mathcal{P}} n_\mfp(\mfm) \deg \mfp.$$ 
By \cite[Lemma 5.5]{mr}, for any integer $n \geq 0$, there are finitely many effective divisors of degree $n$. This proves that $\mcp$ is a countable set that satisfies the hypothesis of our main theorems. Let the norm map $N: \mathcal{M} \rightarrow \mathbb{Z}_{>0}$ be the $q$-power map defined as $\mathfrak{m} \mapsto N(\mfm) := q^{\deg \mathfrak{m}}$. Let $X = \{ q^z : z \in \mathbb{Z} \}$. 

By \cite[Lemma 5.8 \& Corollary 4 to Theorem 5.4]{mr}, for a non-negative integer $n$ satisfying $n > 2 G_K - 2$, the number of effective divisors of degree $n$ is 
$$C_K \frac{q^{n- G_K+1} -1}{q-1}.$$ 
Thus, for sufficiently large $n$, we obtain
$$\sum_{\substack{\mfm \\ \deg \mfm \leq n}} 1 = \frac{C_K}{q^{G_K}} \left( \frac{q}{q-1} \right)^2 q^n + O(n).$$
This satisfies condition \eqref{star} with $\kappa = \frac{C_K}{q^{G_K}} \left( \frac{q}{q-1} \right)^2$ and $\theta = \epsilon$ for any $\epsilon \in (0,1)$. Thus, \thmref{erdoskacforomegaA-hfree} and \thmref{erdoskacforomegaA} give the Erdos-Kac theorems for $\omega_\mcA(\mfm)$ over $h$-free and $k$-full effective divisors in a global function field respectively as the following:
\begin{cor}
Let $n, h \in \mathbb{Z}_{>0}$ with $h \geq 2$. Let $K/\mathbb{F}_q$ be a global function field with genus $G_K$ and class number $C_K$. Let $\mathcal{S}_h(n)$ be the set of $h$-free effective divisors in $K$ of degree less than or equal to $n$. Let $\mcA$ be any sequence from Types 1-4 and Type 5 with $k = 1$. Then for $a \in \mathbb{R}$, we have
$$\lim_{n \rightarrow \infty} \frac{1}{|\mathcal{S}_h(n)|} \bigg| \left\{ \mfm \in \mathcal{S}_h(n) \ : \ q^{\deg \mathfrak{m}} \geq 3, 
\ \frac{\frac{1}{a_1} \omega_\mcA(\mfm) - \log \log q^{\deg \mathfrak{m}}}{\sqrt{\log \log q^{\deg \mathfrak{m}}}} \leq a \right\} \bigg| = \Phi(a).$$
\end{cor} 
\begin{cor}
Let $n, h \in \mathbb{Z}_{>0}$. Let $K/\mathbb{F}_q$ be a global function field with genus $G_K$ and class number $C_K$. Let $\mathcal{N}_k(n)$ be the set of $k$-full effective divisors in $K$ of degree less than or equal to $n$. Let $\mcA$ be any sequence from Types 1-5. Then for $a \in \mathbb{R}$, we have
$$\lim_{n \rightarrow \infty} \frac{1}{|\mathcal{N}_k(n)|} \bigg| \left\{ \mfm \in \mathcal{N}_k(n) \ : \ q^{\deg \mathfrak{m}} \geq 3, 
\ \frac{\frac{1}{a_k} \omega_\mcA(\mfm) - \log \log q^{\deg \mathfrak{m}}}{\sqrt{\log \log q^{\deg \mathfrak{m}}}} \leq a \right\} \bigg| = \Phi(a).$$
\end{cor}
\begin{rmk}
For the special case when $K = \mathbb{F}_q(x)$, whose genus and class number are 0 and 1 respectively, we can consider the abelian monoid $Z = \mathbb{F}_q[x]$, the ring of monic polynomials in one variable over $\mathbb{F}_q$. The prime elements of $Z$ are the monic irreducible polynomials in $Z$. The localizations of $Z$ at these prime elements exhaust the set of all primes of $K$ except one, the prime at infinity. Using the fact that there are $q^n$ monic polynomials of degree $n$, we obtain
$$\sum_{\substack{\mfm \in Z \\ \deg \mathfrak{m} \leq n}} 1 = \frac{q}{q-1} q^n + O(1).$$
This satisfies condition \eqref{star} with $\kappa = q/(q-1)$ and $\theta = 0$. Thus, the Erd\H{o}s-Kac theorems for $\omega(\mfm)$ over $h$-free and $h$-full polynomials over finite fields can be deduced from \thmref{erdoskacforomegaA-hfree} and \thmref{erdoskacforomegaA}. Such a result will be equivalent to the ones studied by Lal\'in and Zhang \cite[Theorems 4.2 \& 6.2]{lz}. 

Similarly, we can deduce the Erd\H{o}s-Kac theorems for $\Omega(\mfm)$ over $h$-free and $h$-full polynomials, which will be equivalent to the results of  Lal\'in and Zhang \cite[Theorems 1.3 \& 1.6]{lz}. We can also deduce the Erd\H{o}s-Kac theorems for $\omega_1(\mfm)$ over $h$-free polynomials and $\omega_h(\mfm)$ over $h$-full polynomials, which will be equivalent to the results of  Gom\'ez and Lal\'in \cite[Theorems 1.2 \& 1.6]{lg}. 

Note that, in this special case, $\kappa = q/(q-1)$ instead of $(q/(q-1))^2$ to account for the lack of the prime at infinity analog in its construction.
\end{rmk}
\begin{rmk}
    The study of global function fields over $\mathbb{F}_q$ is geometrically equivalent to the study of irreducible projective varieties of dimension 1 over $\mathbb{F}_q$. Such varieties are also called irreducible curves. We can apply our main theorems to irreducible projective varieties of dimension r over $\mathbb{F}_q$, where $r$ is any positive integer. We study this in the following subsection.
\end{rmk}

\subsection{The case of effective \texorpdfstring{$0$}{}-cycles in geometrically irreducible projective varieties of dimension \texorpdfstring{$r$}{}}

In this subsection, we adopt notation from \cite[Example 4 of Section 4]{liuturan}. 

Let $q$ be a prime power and $\mathbb{F}_q$ be the finite field with $q$ elements. Let $r$ be a positive integer. Let $V/\mathbb{F}_q$ be a geometrically irreducible projective variety of dimension $r$. Let $\mcp$ be the set of closed points of $V/\mathbb{F}_q$, which is in bijection with the set of orbits of $V/\mathbb{F}_q$ under the action of $\textnormal{Gal}(\bar{\mathbb{F}}_q/\mathbb{F}_q)$ (see \cite[Proposition 6.9]{lorenzini}). For each $\mfp \in \mcp$, we define the degree of $\mfp$, $\deg \mfp$, to be the length of the corresponding orbit. Let $\mcm$ be the free abelian monoid generated by $\mcp$. We call elements in $\mcm$ as effective 0-cycles. For $\mfm \in \mcm$, we have $\mfm = \sum_{\mfp \in  \mathcal{P}} n_\mfp(\mfm) \mfp$ with $n_\mfp(\mfm) \in \mathbb{Z}_{>0} \cup \{ 0 \}$ and $n_\mfp(\mfm) =0$ for all but finitely many $\mfp$. We define the degree of $\mfm$ as
$$\deg \mfm = \sum_{\mfp \in  \mathcal{P}} n_\mfp(\mfm) \deg \mfp.$$ 
By \cite[Lemma 3.11]{lorenzini}, we deduce that $\mcp$ is countable and satisfies the hypothesis of our main theorems. Let the norm map $N: \mathcal{M} \rightarrow \mathbb{Z}_{>0}$ be the $q^r$-power map defined as $\mathfrak{m} \mapsto N(\mfm) := q^{r \deg \mathfrak{m}}$. Let $X = \{ q^{rz} : z \in \mathbb{Z} \}$. In \cite[Remark 1 of Section 4]{liuturan}, the third author proved that
$$\sum_{\substack{\mfm \\ \deg \mathfrak{m} \leq n}} 1 = \kappa' \left( \frac{q^r}{q^r-1} \right) q^{rn} + O \left( n \cdot q^{(r-1)n}  \right),$$
where $\kappa'$ is some positive constant defined explicitly in \cite[Lemma 7 of Section 4]{liuturan}. This satisfies condition \eqref{star} with $\kappa = \kappa' \left( \frac{q^r}{q^r-1} \right)$ and $\theta = \epsilon$ for any $\epsilon \in (1-1/r,1)$. Thus, \thmref{erdoskacforomegaA-hfree} and \thmref{erdoskacforomegaA} gives the Erd\H{o}s-Kac theorems $\omega_\mcA(\mfm)$ over $h$-free and $k$-full effective 0-cycles in a geometrically irreducible projective variety of dimension $r$ as the following:

\begin{cor}
Let $r, n \in \mathbb{Z}_{>0}$. Let $h \geq 2$ be an integer. Let $V/\mathbb{F}_q$ be a geometrically irreducible projective variety of dimension $r$. Let $\mathcal{S}_h(n)$ be the set of $h$-free effective $0$-cycles in $V$ of degree less than or equal to $n$.  Let $\mcA$ be any sequence from Types 1-4 and Type 5 with $k = 1$. Then for $a \in \mathbb{R}$, we have
$$\lim_{n \rightarrow \infty} \frac{1}{|\mathcal{S}_h(n)|} \bigg| \left\{ \mfm \in \mathcal{S}_h(n) \ : \ q^{r \deg \mathfrak{m}} \geq 3, 
\ \frac{\frac{1}{a_1} \omega_\mcA(\mfm) - \log \log q^{r \deg \mathfrak{m}}}{\sqrt{\log \log q^{r \deg \mathfrak{m}}}} \leq a \right\} \bigg| = \Phi(a).$$
\end{cor} 
\begin{cor}
Let $r,n \in \mathbb{Z}_{>0}$. Let $k \geq 1$ be an integer. Let $V/\mathbb{F}_q$ be a geometrically irreducible projective variety of dimension $r$. Let $\mathcal{N}_k(n)$ be the set of $k$-full effective $0$-cycles in $V$ of degree less than or equal to $n$. Let $\mcA$ be any sequence from Types 1-5. Then for $a \in \mathbb{R}$, we have
$$\lim_{n \rightarrow \infty} \frac{1}{|\mathcal{N}_k(n)|} \bigg| \left\{ \mfm \in \mathcal{N}_k(n) \ : \ q^{r \deg \mathfrak{m}} \geq 3, 
\ \frac{\frac{1}{a_k} \omega_\mcA(\mfm) - \log \log q^{r \deg \mathfrak{m}}}{\sqrt{\log \log q^{r \deg \mathfrak{m}}}} \leq a \right\} \bigg| = \Phi(a).$$
\end{cor}

\bibliographystyle{plain} 
\bibliography{mybib.bib} 

\end{document}